\documentclass[
reqno]{amsart}
\usepackage{latexsym,amsmath,amssymb,amscd}
\usepackage[all]{xy}
\usepackage{enumerate}

\def\today{\ifcase \month \or
  January \or February \or March \or April \or
  May \or June \or July \or August \or
  September \or October \or November \or December \fi
  \space\number\day , \number\year}

\makeatletter
  \newcommand\@dotsep{4.5}
  \def\@tocline#1#2#3#4#5#6#7{\relax
    \ifnum #1>\c@tocdepth 
    \else
    \par \addpenalty\@secpenalty\addvspace{#2}%
    \begingroup \hyphenpenalty\@M
    \@ifempty{#4}{%
    \@tempdima\csname r@tocindent\number#1\endcsname\relax
        }{%
        \@tempdima#4\relax
          }%
      \parindent\z@ \leftskip#3\relax \advance\leftskip\@tempdima\relax
      \rightskip\@pnumwidth plus1em \parfillskip-\@pnumwidth
      #5\leavevmode\hskip-\@tempdima #6\relax
      \leaders\hbox{$\m@th
      \mkern \@dotsep mu\hbox{.}\mkern \@dotsep mu$}\hfill
      \hbox to\@pnumwidth{\@tocpagenum{#7}}\par
      \nobreak
        \endgroup
        \fi}
\makeatother

\begin{document}

\makeatletter
\@addtoreset{figure}{section}
\def\thefigure{\thesection.\@arabic\c@figure}
\def\fps@figure{h,t}
\@addtoreset{table}{bsection}

\def\thetable{\thesection.\@arabic\c@table}
\def\fps@table{h, t}
\@addtoreset{equation}{section}
\def\theequation{
\arabic{equation}}
\makeatother

\newcommand{\bfi}{\bfseries\itshape}

\newtheorem{theorem}{Theorem}
\newtheorem{acknowledgment}[theorem]{Acknowledgment}
\newtheorem{algorithm}[theorem]{Algorithm}
\newtheorem{axiom}[theorem]{Axiom}
\newtheorem{case}[theorem]{Case}
\newtheorem{claim}[theorem]{Claim}
\newtheorem{conclusion}[theorem]{Conclusion}
\newtheorem{condition}[theorem]{Condition}
\newtheorem{conjecture}[theorem]{Conjecture}
\newtheorem{construction}[theorem]{Construction}
\newtheorem{corollary}[theorem]{Corollary}
\newtheorem{criterion}[theorem]{Criterion}
\newtheorem{data}[theorem]{Data}
\newtheorem{definition}[theorem]{Definition}
\newtheorem{example}[theorem]{Example}
\newtheorem{lemma}[theorem]{Lemma}
\newtheorem{notation}[theorem]{Notation}
\newtheorem{problem}[theorem]{Problem}
\newtheorem{proposition}[theorem]{Proposition}
\newtheorem{question}[theorem]{Question}
\newtheorem{remark}[theorem]{Remark}
\newtheorem{setting}[theorem]{Setting}
\numberwithin{theorem}{section}
\numberwithin{equation}{section}

\newcommand{\todo}[1]{\vspace{5 mm}\par \noindent
\framebox{\begin{minipage}[c]{0.85 \textwidth}
\tt #1 \end{minipage}}\vspace{5 mm}\par}

\renewcommand{\1}{{\bf 1}}

\newcommand{\hotimes}{\widehat\otimes}

\newcommand{\Ad}{{\rm Ad}}
\newcommand{\Alt}{{\rm Alt}\,}
\newcommand{\Ci}{{\mathcal C}^\infty}
\newcommand{\comp}{\circ}
\newcommand{\wt}{\widetilde}

\newcommand{\C}{\text{\bf C}}
\newcommand{\D}{\text{\bf D}}
\newcommand{\Hb}{\text{\bf H}}
\newcommand{\N}{\text{\bf N}}
\newcommand{\R}{\text{\bf R}}
\newcommand{\T}{\text{\bf T}}
\newcommand{\Ub}{\text{\bf U}}

\newcommand{\ph}{\text{\bf P}}
\newcommand{\Cartan}{{\rm Cartan}}
\newcommand{\de}{{\rm d}}
\newcommand{\Decomp}{{\rm Decomp}}
\newcommand{\diag}{{\rm diag}}
\newcommand{\ev}{{\rm ev}}
\newcommand{\fimes}{\mathop{\times}\limits}
\newcommand{\fin}{{\rm fin}}
\newcommand{\id}{{\rm id}}
\newcommand{\ie}{{\rm i}}
\newcommand{\End}{{\rm End}\,}
\newcommand{\Gr}{{\rm Gr}}
\newcommand{\GL}{{\rm GL}}
\newcommand{\Hilb}{{\bf Hilb}\,}
\newcommand{\Hom}{{\rm Hom}}
\renewcommand{\Im}{{\rm Im}}
\newcommand{\Ker}{{\rm Ker}\,}
\newcommand{\Lie}{\textbf{L}}
\newcommand{\lf}{{\rm l}}
\newcommand{\pr}{{\rm pr}}
\newcommand{\Ran}{{\rm Ran}\,}
\newcommand{\rank}{{\rm rank}\,}
\newcommand{\sa}{{\rm sa}}
\newcommand{\spann}{{\rm span}}
\newcommand{\spec}{{\rm spec}}

\newcommand{\Tr}{{\rm Tr}\,}
\newcommand{\Tran}{\textbf{Trans}}

\newcommand{\CC}{{\mathbb C}}
\newcommand{\NN}{{\mathbb N}}
\newcommand{\RR}{{\mathbb R}}
\renewcommand{\SS}{{\mathbb S}}
\newcommand{\TT}{{\mathbb T}}

\newcommand{\bb}{b}
\newcommand{\vv}{v}
\newcommand{\ww}{w}

\newcommand{\G}{{\rm G}}
\newcommand{\U}{{\rm U}}
\newcommand{\Gl}{{\rm GL}}
\newcommand{\SL}{{\rm SL}}
\newcommand{\SU}{{\rm SU}}
\newcommand{\VB}{{\rm VB}}

\newcommand{\Ac}{{\mathcal A}}
\newcommand{\Bc}{{\mathcal B}}
\newcommand{\Cc}{{\mathcal C}}
\newcommand{\Dc}{{\mathcal D}}
\newcommand{\Ec}{{\mathcal E}}
\newcommand{\Fc}{{\mathcal F}}
\newcommand{\Gc}{{\mathcal G}}
\newcommand{\Hc}{{\mathcal H}}
\newcommand{\Ic}{{\mathcal I}}
\newcommand{\Jc}{{\mathcal J}}
\newcommand{\Kc}{{\mathcal K}}
\newcommand{\Lc}{{\mathcal L}}
\renewcommand{\Mc}{{\mathcal M}}
\newcommand{\Nc}{{\mathcal N}}
\newcommand{\Oc}{{\mathcal O}}
\newcommand{\Pc}{{\mathcal P}}
\newcommand{\Qc}{{\mathcal Q}}
\newcommand{\Rc}{{\mathcal R}}
\newcommand{\Sc}{{\mathcal S}}
\newcommand{\Tc}{{\mathcal T}}
\newcommand{\Uc}{{\mathcal U}}
\newcommand{\Vc}{{\mathcal V}}
\newcommand{\Xc}{{\mathcal X}}
\newcommand{\Yc}{{\mathcal Y}}
\newcommand{\Zc}{{\mathcal Z}}
\newcommand{\Ag}{{\mathfrak A}}
\newcommand{\Bg}{{\mathfrak B}}
\renewcommand{\gg}{{\mathfrak g}}
\newcommand{\hg}{{\mathfrak h}}
\newcommand{\mg}{{\mathfrak m}}
\newcommand{\ug}{{\mathfrak u}}
\newcommand{\nng}{{\mathfrak n}}
\newcommand{\pg}{{\mathfrak p}}
\newcommand{\Gg}{{\mathfrak G}}
\newcommand{\Hg}{{\mathfrak H}}
\newcommand{\Ig}{{\mathfrak I}}
\newcommand{\Jg}{{\mathfrak J}}
\newcommand{\Lg}{{\mathfrak L}}
\newcommand{\Sg}{{\mathfrak S}}
\newcommand{\Ug}{{\mathfrak u}}

\markboth{}{}

\makeatletter
\title[Cartan subalgebras of operator ideals]{Cartan subalgebras of operator ideals}
\author[Daniel Belti\c t\u a]{Daniel Belti\c t\u a$^*$}
\thanks{$^*$
Partially supported by the Grant
of the Romanian National Authority for Scientific Research, CNCS-UEFISCDI,
project number PN-II-ID-PCE-2011-3-0131.
}
\author[Sasmita Patnaik]{Sasmita Patnaik$^{**}$}
\thanks{$^{**}$Partially supported by the Graduate Dissertation Fellowship
from the Charles Phelps Taft Research Center.}
\author[Gary Weiss]{Gary Weiss$^{***}$}
\thanks{$^{***}$
Partially supported by
Simons Foundation Collaboration Grant for Mathematicians \#245014 and several Taft Research Center travel grants.}
\date{August 21, 2014}


\makeatother

\begin{abstract}
Denote by $U_{\Ic}(\Hc)$ the group of all unitary operators in $\1+\Ic$  
where $\Hc$ is a separable infinite-dimensional complex Hilbert space and  
$\Ic$ is any two-sided ideal of $\Bc(\Hc)$. 
A Cartan subalgebra $\Cc$ of $\Ic$ 
is defined in this paper as a maximal abelian self-adjoint subalgebra of~$\Ic$  
and its conjugacy class is defined herein as the set of Cartan subalgebras $\{V\Cc V^*\mid V\in U_{\Ic}(\Hc)\}$.
For nonzero proper ideals $\Ic$ we construct an uncountable family
of Cartan subalgebras of $\Ic$ with distinct conjugacy classes. 
This is in contrast
to the by now classical observation of P. de La Harpe 
who noted 
that 
when $\Ic$ is any of the Schatten ideals,  
there is precisely
one conjugacy class under the action of the full group of unitary operators on~$\Hc$. 
Our perspective is that the action of the full unitary group on Cartan subalgebras of $\Ic$ is transitive, 
while by shrinking to $U_{\Ic}(\Hc)$ we obtain an action with uncountably many orbits 
if $\{0\}\ne\Ic\ne\Bc(\Hc)$. 

In the case when $\Ic$ is a symmetrically normed ideal and
 is the dual of some Banach space,
we show how the conjugacy classes of the Cartan subalgebras of $\Ic$ become 
smooth manifolds modeled on suitable Banach spaces.  
These manifolds are endowed with groups of smooth transformations given by the action of the group $U_{\Ic}(\Hc)$ 
on the orbits,  
and are equivariantly diffeomorphic to each other.  
We then find that there exists a unique diffeomorphism class of full flag manifolds of $U_{\Ic}(\Hc)$ 
and we give its construction.  
This resembles the case of compact Lie groups when one has a unique full flag manifold, 
since all the Cartan subalgebras are conjugated to each other.
\end{abstract}

\maketitle

\tableofcontents

\section{Introduction}

The correspondence between Lie groups and Lie algebras has been extended far beyond 
the classical setting of finite-dimensional Lie groups   
(see for instance \cite{Ne06}).  
Lie theory consequently impacted various areas of mathematics, 
in particular representation theory and operator algebras, 
and this interaction contributed back to further developments in Lie theory itself. 
We will next illustrate these general remarks  
by a very few other references  
in order to explain the motivation for our present investigation on Cartan subalgebras of operator ideals   
that is a sequel to our introductory survey paper \cite{BPW13b}. 

Cartan subalgebras play a central role in 
the structure theory of finite-dim\-ensio\-nal Lie algebras, 
in part by extending the role of diagonal subalgebras to the study of matrix algebras. 
Consequently there has been a continuing interest  
in finding the appropriate notion of Cartan subalgebras of various infinite-dimensional Lie algebras.  
And this has led to progress in various directions 
of which we will briefly mention only three that are more relevant for this paper: 
\begin{itemize}
\item Direct limit groups: representation theory 
based on root-space decompositions associated to Cartan subalgebras 
\cite{NRW01}, \cite{DPW02}, \cite{Wo05}.  
\item Classical Banach-Lie groups associated with the Schatten classes: 
structure and representation theory 
\cite{dlH72}, \cite{Bo80}, \cite{Ne98}, \cite{Ne02}, \cite{Ne04}, \cite{AV07}, \cite{ALR10}, \cite{CD13}. 
\item Maximal abelian self-adjoint subalgebras of von Neumann algebras and of $C^*$-algebras 
\cite{Re08}, \cite{SS08}. 
\end{itemize} 
The relevance of the first direction is obvious from the fact that many direct limit groups 
are closely related to the classical groups associated with the ideal of finite-rank operators. 
The relevance of the second direction is seen for instance in our comments 
following Question~\ref{first_quest} below 
(which first appears in \cite{BPW13b}), 
where we raised the problem of studying the conjugacy classes 
of Cartan subalgebras of infinite-dimensional classical groups. 

There is a huge literature devoted to the third of these research directions, 
hence we merely cite here a nice survey and a book but we will not discuss these citations further. 
We next briefly mention a few facts from this deep and very active theory, 
since in the present paper we will try to find analogs for some of these facts  
in the theory of operator ideals. 

\textit{Cartan subalgebras of proper operator ideals} of $\Bc(\Hc)$ as defined here 
are simply the maximal abelian self-adjoint subalgebras of these ideals. 
This is a topology-free notion since some operator ideals may not have any reasonable complete linear topology. 

However, an additional topological regularity property is required for defining 
the Cartan subalgebras of $\Bc(\Hc)$ 
and more generally for Cartan subalgebras of any $C^*$-algebra or of any von Neumann algebra 
(see \cite[Def. 2.1, 5.1]{Re08}). 
Let $\Mc$ be any unital $*$-subalgebra of $\Bc(\Hc)$ with 
its unitary group $U_{\Mc}:=\{u\in\Mc\mid uu^*=u^*u=\1\}$. 
For any maximal abelian self-adjoint subalgebra (masa) $\Cc\subseteq\Mc$, 
its \textit{normalizer} is a subgroup of $U_{\Mc}$ that is commonly defined 
as 
$$U_{\Mc,\Cc}:=\{u\in U_{\Mc}\mid u\Cc u^*=\Cc\}.$$ 
Observe that $U_{\Mc,\Cc}\supseteq U_{\Mc}\cap\Cc$. 
If $U_{\Mc,\Cc}\subset\Cc$, then $U_{\Mc,\Cc}=U_{\Mc}\cap\Cc$, 
and we say that $\Cc$ is \emph{singular} since its normalizer is minimum possible. 
To describe the opposite situation, which should be thought of as a regularity property of $\Cc$ 
(see \eqref{reg} below), 
we single out two particularly important cases: 
\begin{enumerate}[(i)]
\item\label{tau_1} If $\Mc$ is closed in the operator norm topology, 
that is, $\Mc$ is a concrete $C^*$-algebra, 
then $\Cc$ is also closed in the operator norm topology by its maximality property. 
\item\label{tau_2} Similarly, $\Mc$ is closed in the weak operator topology, 
that is, $\Mc$ is a von Neumann algebra,  
then $\Cc$ is weakly closed again by its maximality property. 
\end{enumerate}
Let us denote by $\tau$ the operator norm topology of $\Mc$ in~\eqref{tau_1} 
and the weak operator topology in~\eqref{tau_2}. 
One calls the masa $\Cc$ a \textit{Cartan subalgebra} of $\Mc$ if 
there exists a faithful conditional expectation $E\colon\Mc\to\Cc$, 
(that is, $E$ is $\tau$-continuous, $E^2=E$, $\Vert E\Vert=1$, $E(\Mc)=\Cc$  
and if $0\le x\in\Ker E$ then $x=0$) and $\Cc$ 
has the following \emph{regularity property}: 
\begin{equation}\label{reg}
\text{The normalizer $U_{\Mc,\Cc}$ spans a $\tau$-dense linear subspace of $\Mc$.} 
\end{equation}
(See \cite[Defs. 2.1, 5.1]{Re08} and note that $\1\in\Cc$ by the maximality property of $\Cc$.)

If $\Mc=M_n$ is the algebra of complex $n\times n$ matrices, then every masa of $M_n$ is a Cartan subalgebra 
in the above sense.  
And in particular this is the case for its canonical diagonal subalgebra $D_n\subset M_n$, 
since for instance $U_{M_n,D_n}$ contains all the permutation matrices.  
However, in the case of the infinite-dimensional von Neumann algebras, 
it was discovered very early \cite{Di54} 
that there may exist singular masa's, 
and such singular masa's were actually found in every separable II$_1$ factor \cite{Po83}.  

As regards these Cartan subalgebras, 
 their existence has deep implications in the structure of von Neumann algebras \cite{FM77}. 
However, there exist II$_1$ factors that do not have Cartan subalgebras at all, 
for instance the von Neumann algebras of the free groups with finitely many generators \cite{Vo96}.  
On the other hand, there exist II$_1$ factors, like the finite hyperfinite factor, 
for which there exist uncountably many conjugacy classes of Cartan subalgebras \cite{FM77}, \cite{Par85}, 
where we mean conjugacy by transformations $x\mapsto uxu^*$ defined by unitaries~$u$. 
Also II$_1$ factors have been exhibited for which all the Cartan subalgebras are conjugated 
to each other~\cite{OP10}.

\subsection*{Description of the present paper}
We investigate here conjugacy properties of Cartan subalgebras 
of 
operator ideals of $\Bc(\Hc)$ on separable Hilbert space~$\Hc$. 
(All ideals herein are considered two-sided.) 
If we consider the full unitary group $U(\Hc)=\{V\in\Bc(\Hc)\mid V^*V=VV^*=\1\}$,
then to any nonzero operator ideal $\Ic\subsetneqq\Bc(\Hc)$
there corresponds 
the 
subgroup of $U(\Hc)$:
$$U_{\Ic}(\Hc):=U(\Hc)\cap(\1+\Ic)\subsetneqq U(\Hc).$$
Then, as for $U(\Hc)=U_{\Ic}(\Hc)$ when $\Ic=\Bc(\Hc)$,  
the adjoint action $\Ad(V)\Cc=V\Cc V^*$,  
for these unitaries $V\in U_{\Ic}(\Hc)$ in particular, preserves the class of Cartan subalgebras.  
That is,  
$V\Cc V^*\subset \Ic$ is again a Cartan subalgebra for every Cartan subalgebra $\Cc\subset\Ic$. 
The group $U_{\Ic}(\Hc)$ thus acts by unitary equivalence on the set of all Cartan subalgebras 
of $\Ic$. 
We prove that there exist uncountably many 
distinct 
orbits of this group action (Theorem~\ref{uncountable}).
(Recall the obvious fact that 
distinct orbits are actually disjoint.) 
This is strikingly different from the well-known situation 
with $\Ic=\Bc(\Hc)$
when one has countably many orbits  
\cite[Rem. 4.2]{BPW13b}, 
or with $\dim\Hc<\infty$ 
when the aforementioned action of group $U_{\Ic}(\Hc)$ 
on the set of all Cartan subalgebras of $\Ic$  
is transitive   
(meaning that there exists only one orbit; see Theorem~\ref{first_th} and the discussion preceding it). 
For a wide class of complete norm ideals, including the Schatten ideals $\Sg_p(\Hc)$ with $1\le p<\infty$,
for each of these orbits we exhibit in Theorem~\ref{smooth} a natural structure of a smooth Banach manifold which turns it 
into a smooth $U_{\Ic}(\Hc)$-homogeneous space 
and the group $U_{\Ic}(\Hc)$ 
is actually a Banach-Lie group.
The smooth homogeneous spaces corresponding to the various orbits turn out to be 
diffeomorphic to each other, and so 
there again exists an essentially unique full flag manifold
just as in the case of finite-dimensonal semisimple Lie groups.
We also obtain 
some results on the size and shape of the union of all Cartan subalgebras
from a fixed $U_{\Ic}(\Hc)$-orbit. 
(See Section~\ref{u-diag}, the first part.) 

By way of 
explaining the motivation for the present investigation, 
we  
recall what happens
in the finite-dimensional case 
where the only nonzero operator ideals are just full matrix algebras~$M_n(\CC)$. 
Any two maximal abelian self-adjoint subalgebras of the matrix algebra $M_n(\CC)$
are mapped to each other by the unitary equivalence 
$\Cc \mapsto V\Cc V^*$ 
for a suitably  
chosen $V\in U(n)$, as a direct consequence of the spectral theorem for normal matrices.
This observation can be thought of as a statement on the compact Lie group $G=U(n):=\{V\in M_n(\CC)\mid V^*V=\1\}$,
whose Lie algebra is $\gg=\ug(n):=\{X\in M_n(\CC)\mid X^*=-X\}$. 
The complexification of the latter algebra is $\Gg=M_n(\CC)$
viewed as a Lie algebra with the usual Lie bracket $[X,Y]=XY-YX$ for all $X,Y\in M_n(\CC)$,
and the unitary equivalence transform is the adjoint action $\Ad_G(V)X=VXV^{-1}$ whenever $V\in G$ and $X\in\Gg$.
In this context, the above statement on the compact Lie group $G=U(n)$ 
actually holds true for any compact Lie group. 
This is known as the conjugacy theorem for 
Cartan subalgebras  
and can be stated as follows:  

\begin{theorem}\label{first_th}
Let $G$ be a compact Lie group whose Lie algebra is $\gg$. 
If the complexified Lie algebra $\Gg=\gg_{\CC}:=\gg\otimes_{\RR}\CC=\gg\dotplus\ie \gg$
is endowed with the involution
given by
$(Y+\ie Z)^*=-Y+\ie Z$ for $Y,Z\in\gg$,
then any two Cartan subalgebras $\Cc_1$ and $\Cc_2$ of $\Gg$ are
{\bfi $G$-conjugated} to each other.
That is, there exists
$g\in G$ such that $\Ad_G(g)\Cc_1=\Cc_2$,
where $\Ad_G\colon G\times \Gg\to\Gg$ stands for the adjoint action of the Lie group~$G$. 
\end{theorem}

\begin{proof}
See for instance \cite[Th. 4.34]{Kn02}.
\end{proof}

The aim of this paper is to extend this study to operator ideals 
and 
their corresponding  
smaller unitary groups  
via

\begin{question}\label{first_quest}
\normalfont
To what extent 
does 
Theorem~\ref{first_th}
hold
true when the above finite-dimensional Lie groups and Lie algebras are replaced by
$$\Gg=\Ic\text{ and } G=U_{\Ic}(\Hc):=U(\Hc)\cap(\1+\Ic)$$
where $\Hc$ is a separable infinite-dimensional complex Hilbert space,
$U(\Hc)$ is the full unitary group on~$\Hc$, and $\Ic$  is an operator ideal in~$\Bc(\Hc)$?
\end{question}
 
The point is that here we instead investigate conjugacy results involving the smaller unitary group~$U_{\Ic}(\Hc)$
rather than the full unitary group $U(\Hc)$.
We recall that the variant of the above question with $G=U(\Hc)$ was addressed in \cite[page 33]{dlH72}
for the ideal of finite-rank operators,
and in \cite[\S II.4, Props. 10 and 12]{dlH72} 
when $\Gg=\Ic$ is any of the Schatten ideals. 
That argument based on the spectral theorem actually carries over directly
to fully general proper operator ideals leading to the following 
infinite-dimensional version of Theorem~\ref{first_th}:

\begin{remark}\label{spectral}
\normalfont
For every operator ideal $\Ic\subsetneqq\Bc(\Hc)$  
and any two Cartan subalgebras $\Cc_1$ and $\Cc_2$ of $\Ic$ 
one has $V\Cc_1V^*=\Cc_2$ for some $V\in U(\Hc)$. 
(See also Remark~\ref{equivar}\eqref{equivar_item1} below for a proof of this fact.) 
\end{remark}

The answer to Question~\ref{first_quest} is obvious if $\Ic=\{0\}$ and
is also well known in the case $\Ic=\Bc(\Hc)$ 
(the relevant facts are recalled in \cite[Rem. 4.2]{BPW13b}). 
Let us also note that problems similar 
to Question~\ref{first_quest}
could be raised in connection with the other classical groups associated with operator ideals,
and more generally about various infinite-dimensional versions of reductive Lie groups
(see \cite{dlH72}, \cite{Be11}, \cite{BPW13b}, and also \cite{KMRT98} for their finite-dimensional versions). 
That is, if $G$ is one of these groups with its corresponding Lie algebra $\gg$, 
it would be interesting to investigate the $G$-conjugacy classes of 
the Cartan subalgebras (maximal abelian self-adjoint subalgebras) of $\gg$.  
This problem makes sense since $G$ is a group of invertible operators  
and it acts on the operator Lie algebra $\gg$ 
by the group action $G\times\gg\to\gg$, $(V,X)\mapsto VXV^{-1}$.

\section{Preliminaries}\label{Prelims}

\subsection*{Operator ideals on Hilbert spaces}
Throughout this paper we denote by $\Hc$ a separable infinite-dimensional complex Hilbert space
with a fixed orthonormal basis $\bb=\{\bb_n\}_{n\ge 1}$,
and by $\Dc$ the corresponding set of diagonal operators in $\Bc(\Hc)$.
If $\lambda=\langle\lambda_n\rangle_{n=1}^\infty$ is a bounded sequence of complex numbers,
then we denote
by $\diag\,\lambda\in\Dc$ the operator whose sequence of diagonal entries is $\lambda$.

Let $\Kc(\Hc)$ denote the ideal of compact operators on $\Hc$ and 
$\Fc(\Hc)$ denote the ideal of finite-rank operators on $\Hc$.
For every $F\in\Fc(\Hc)$ we denote its rank by $\rank F:=\dim F(\Hc)$, 
the dimension of the image of~$F$.

We will use the 
notation for the full unitary group on~$\Hc$:
$$U(\Hc)=\{V\in\Bc(\Hc)\mid V^*V=VV^*=\1\}.$$
By 
\textit{operator ideal} in $\Bc(\Hc)$, or \emph{$\Bc(\Hc)$-ideal}, 
we will always mean a two-sided ideal of the algebra $\Bc(\Hc)$
of all bounded linear operators on~$\Hc$,
and we say $\Ic$ is a proper ideal if $\{0\}\subsetneqq\Ic\subsetneqq\Bc(\Hc)$. 
Recall $\Fc(\Hc) \subseteq \Ic \subseteq \Kc(\Hc)$ for every proper ideal $\Ic$ of $\Bc(\Hc)$. 
For an operator ideal $\Ic$ we denote its real part
$\Ic^{\sa}:=\{T\in\Ic\mid T=T^*\}$
and its positive cone 
$\Ic^{+}:=\{T\in\Ic\mid T\ge 0\}$.
To each ideal $\Ic$ in $\Bc(\Hc)$ there corresponds the group of unitary operators
$$U_{\Ic}(\Hc):=U(\Hc)\cap(\1+\Ic).$$
Thus, if $\Ic=\Bc(\Hc)$ then $U_{\Ic}(\Hc)=U(\Hc)$,
while for $\Ic=\Kc(\Hc)$ 
we denote $U_{\Ic}(\Hc)=:U_{\Kc(\Hc)}$. 
Likewise for $\Ic=\Sg_2(\Hc)$ we denote $U_{\Ic}(\Hc)=:U_{\Sg_2(\Hc)}$. 

We have the following spectral characterizations of the operators belonging to the unitary groups
$U(\Hc)$ and $U_{\Ic}(\Hc)$ for $\Ic\subsetneqq\Bc(\Hc)$.

\begin{proposition}\label{sas_prop}
Let $A\in\Bc(\Hc)$ 
and denote by $\TT$ the unit circle in the complex plane.
Then 
\begin{enumerate}[(i)]
 \item $\1+A\in U(\Hc)$ if and only if
$A$ is a normal operator with 
its 
spectrum contained in the circle $-1+\TT$, 
\item $\1+A\in U_{\Ic}(\Hc)$ if and only if $A$ is compact normal with spectrum 
(that is, point spectrum) $\{e^{\ie \theta_k}-1\}_{k\ge 1}$ where $-\pi<\theta_k\le\pi$, $\vert\theta_k\vert\downarrow 0$, 
and $\langle\vert e^{\ie \theta_k}-1\vert\rangle_{k=1}^\infty\in\Sigma(\Ic)$, 
the characteristic set of the ideal~$\Ic$ 
(see Remark~\ref{recall_DFWW04}).
\end{enumerate} 
\end{proposition}

\begin{proof}
If $\1+A\in U(\Hc)$, then $\spec(\1+A)\subseteq\TT$ hence $\spec(A)\subseteq-1+\TT$,
and on the other hand $(\1+A)^*(\1+A)=(\1+A)(\1+A)^*=\1$, which implies $A^*A=AA^*$.
(Another way to prove normality is to notice that $A=(\1+A)-\1$ is the difference of two commuting normal operators,
and then one could use either the spectral theorem for normal operators or even the Fuglede commutativity theorem.)

We now prove the converse assertion. 
Since A is normal, $\1 + A$ is also normal, 
so by spectral theorem for normal operators 
$\1 + A$ is unitarily equivalent
to a multiplication operator $M_\varphi$ on some $L^2$-space, where $\varphi$ is an $L^\infty$-function. 
The hypothesis ensures that the spectrum of $\1+A$ is contained in $\TT$, 
hence $\vert\phi\vert=1$ almost everywhere. 
Hence $M_\varphi$ is unitary, which further implies that $\1 + A$ is a unitary operator.
\end{proof}

\begin{remark}\label{recall_DFWW04}
\normalfont
We collect here for later use a few facts on operator ideals;
see \cite[Sect. 4]{DFWW04} for more details on this terminology.
For any integer $n\ge 1$ and any $T\in\Bc(\Hc)$, 
its $n$-th singular number is defined by $s_n(T)=\inf\{\Vert T-F\Vert\mid F\in\Bc(\Hc),\ \rank F < n\}$
and we denote $s(T):=\langle s_n(T)\rangle_{n=1}^\infty$.
Thus $\Vert T\Vert=s_1(T)\ge s_2(T)\ge\cdots$, and $T$ is a compact operator
if and only if $\lim\limits_{n\to\infty}s_n(T)=0$. 

Moreover one has the well-known characterization of $\Bc(\Hc)$-ideals in terms of characteristic sets 
$$\Sigma(\Ic):=\{s(T)\mid T\in\Ic\} $$
with map $\Ic\mapsto\Sigma(\Ic)$ an inclusion-preserving lattice isomorphism 
between the lattice of $\Bc(\Hc)$-ideals and the lattice of characteristic sets 
(for a modern reference see \cite[Sect. 4]{DFWW04}). 
In particular, $\Sigma(\Kc(\Hc))=c_0^*$, the cone of decreasing to zero positive sequences,  
and $\Sigma(\Fc(\Hc))$ are those $c_0^*$ sequences that also are finitely supported. 

An operator ideal $\Ic$ in $\Bc(\Hc)$  is called \textit{arithmetic mean closed}
if it is closed under $s(\cdot)$-majorization, i.e., 
it satisfies the condition:
if $T\in\Bc(\Hc)$ where for some $K\in\Ic$ we have
$$(\forall n\ge 1)\quad s_1(T)+\cdots+s_n(T)\le s_1(K)+\cdots+s_n(K),$$
then necessarily $T\in\Ic$.
Here are some important properties of these ideals:
\begin{enumerate}[(i)]
\item\label{recall_DFWW04_item1}
If $\Ic$ is arithmetic mean closed  and  $\{P_n\}_{1\le n<N}$ is a family of mutually orthogonal projections in $\Hc$,
where $1\le N\le\infty$, then for every $T\in\Ic$ we have $\sum\limits_{1\le n<N}P_nTP_n\in\Ic$
as a direct consequence of \cite[Ch. II, Th. 5.1]{GK69}.
\item\label{recall_DFWW04_item2}
Conversely, if $\{P_n\}_{1\le n<N}$ is a family of mutually orthogonal rank-one projections with $\sum\limits_{1\le n<N}P_n=\1$
and for every $T\in\Ic$ we have $\sum\limits_{1\le n<N}P_nTP_n\in\Ic$, then the ideal $\Ic$ is arithmetic mean closed
by \cite[Th. 4.5]{KW12}.
\item\label{recall_DFWW04_item3}
Let $\Phi$ be a symmetric norming function as in \cite[Ch.III, \S 3]{GK69}.
In our notation, this equivalently can be defined as a norm on the space of finite-rank diagonal operators
$\Phi\colon\Dc\cap\Fc(\Hc)\to[0,\infty)$ such that $\Phi(P)=1$ if $P$ is a rank-one orthogonal projection
and $\Phi(V_\sigma DV_\sigma^{-1})=\Phi(\vert D\vert)$ for all $D\in\Dc\cap\Fc(\Hc)$ and $\sigma\in\SS_\infty$ 
(see ``Permutation groups'' below for definitions of $\SS_\infty$ and $V_\sigma$).

Let $\Sg_\Phi$ be the set of all operators $T\in\Kc(\Hc)$ such that
$$\Vert T\Vert_\Phi:=\sup_{n\ge 1}\Phi(\diag(s_1(T),\dots,s_n(T),0,0,\dots))<\infty.$$
Then $\Sg_\Phi$ is an operator ideal endowed with the complete (symmetric) norm $\Vert\cdot\Vert_\Phi$
by \cite[Ch. III, Th. 4.1]{GK69},
and this ideal is arithmetic mean closed
(see \cite[subsect. 4.9]{DFWW04} or the Dominance Property in \cite[Ch. III, \S 4]{GK69}).
In particular, for $\Phi(\cdot)=\Vert\cdot\Vert_{\ell^p}$ 
we obtain that the Schatten ideal $\Sg_p(\Hc)$ is arithmetic mean closed whenever $1\le p\le\infty$.
Note also that for $\Phi(\cdot)=\Vert\cdot\Vert_\infty$, $\Sg_\Phi=\Sg_\infty(\Hc)=\Kc(\Hc)$.
\end{enumerate}
\qed
\end{remark}

\subsection*{Permutation groups}
We denote by $\SS_\infty$ the group of all permutations of the positive integers $\{1,2,\dots\}$
and by $\SS_\fin$ its subgroup consisting of the ``finite permutations'',
that is, the permutations that leave fixed all but finitely many natural numbers.
There is a group homomorphism
$$\SS_\infty\to U(\Hc),\quad \sigma\mapsto V_\sigma,$$
which depends on the choice of the orthonormal basis $\bb$,
and is defined by $V_\sigma(\bb_n)=\bb_{\sigma^{-1}(n)}$ whenever $n\in\NN$ and $\sigma\in\SS_{\infty}$. 
Clearly also then $V_\sigma^{-1}=V_{\sigma^{-1}}$. 

\begin{remark}
\normalfont
If $D\in\Dc$, then for every $\sigma\in\SS_\infty$
we have $V_\sigma D V_{\sigma}^{-1}\in\Dc$,
and the diagonal entries of the latter diagonal operator are precisely the diagonal entries of $D$
permuted according to~$\sigma$.
\qed
\end{remark}

\begin{proposition}\label{S_fin}
If $\Ic$ is a proper ideal of $\Bc(\Hc)$ and $\sigma\in\SS_\infty$,
then $V_\sigma\in \1+\Ic$ if and only if $\sigma\in\SS_{\fin}$.
\end{proposition}

\begin{proof}
It is clear that if $\sigma\in\SS_{\fin}$ then $V_\sigma\in \1+\Fc(\Hc)\subseteq\1+\Ic$
since, as is well-known, all proper ideals of $\Bc(\Hc)$ contain the ideal of finite rank operators.
Conversely, assume that for $\sigma\in\SS_\infty$ and  $K\in\Ic$ we have $V_\sigma=\1+K$.
Since $\Ic\ne\Bc(\Hc)$, it follows that $K$ is a compact operator,
hence $\lim\limits_{n\to\infty}\Vert K\bb_n\Vert=0$.
On the other hand for every $n\in\NN$ we have
$$\Vert K\bb_n\Vert=\Vert (V_\sigma-\1)\bb_n\Vert=\Vert \bb_{\sigma^{-1}(n)}-\bb_n\Vert
=\begin{cases}
\sqrt{2} &\text{ if }\sigma(n)\ne n,\\
0 &\text{ if }\sigma(n)=n.
\end{cases}$$
Therefore the condition $\lim\limits_{n\to\infty}\Vert K\bb_n\Vert=0$ implies 
$\sigma(n)=n$ eventually, that is, 
$\sigma\in\SS_{\fin}$.
\end{proof}

\section{$U_{\Ic}(\Hc)$-diagonalization}\label{u-diag}

The aim of this section is to investigate the 
\textit{$U_{\Ic}(\Hc)$-diagonalizable operators in~$\Ic$},
that is, investigate the set
$$\Dc_{\Ic}:=\{VDV^*\mid D\in\Dc\cap\Ic,\, V\in U_{\Ic}(\Hc)\}
=\bigcup_{V\in U_{\Ic}(\Hc)}V(\Dc\cap\Ic)V^*\subseteq\Ic$$
for  $\Ic$ any operator ideal in $\Bc(\Hc)$ and recalling that $\Dc$ is the set of all diagonal operators 
with respect to the fixed basis $\bb=\{\bb_n\}_{n\ge1}$. 
Here we have the union of the sets in 
the $U_{\Ic}(\Hc)$-conjugacy class of the Cartan subalgebra $\Dc\cap\Ic$ of $\Ic$ 
(see also Proposition~\ref{descr} and Definition~\ref{action} below).
This is a set of normal operators in $\Ic$ and
we will also consider its self-adjoint part
\begin{equation}\label{Dsa}
\Dc_{\Ic}^{\sa}:=\Dc_{\Ic}\cap\Ic^{\sa}=\{VDV^*\mid D=D^*\in\Dc\cap\Ic,\, V\in U_{\Ic}(\Hc)\}.
\end{equation} 
It follows by \cite[Prop. 4.3(1)]{BPW13b}
that
\begin{equation}\label{df}
\text{ if }\Ic=\Fc(\Hc)\text{ then }\Dc_{\Ic}^{\sa}=\Ic^{\sa}, 
\end{equation}
but we will prove that this latter equality fails to be true
for any other nontrivial ideal (see Corollary~\ref{conj} below). 

We take this opportunity to offer alternative notions of 
restricted diagonalizability:  
$$\Dc_{\Ic,\Jc}:=\{VDV^*\mid D\in\Dc\cap\Jc,\, V\in U_{\Ic}(\Hc)\}
=\bigcup_{V\in U_{\Ic}(\Hc)}V(\Dc\cap\Jc)V^*\subseteq\Jc.$$
We expect a fair amount of overlap between results in this paper and 
any investigations into these more general notions. 
The difference is essentially the focus on which normal pure point spectrum operators each considers.

\subsection*{General facts {\rm (Proposition~\ref{7april2014}--Corollary~\ref{conj})}}\hfill

We first record the following 
\textit{uniqueness properties of $U_{\Ic}(\Hc)$-dia\-gona\-li\-za\-tion}. 

\begin{proposition}\label{7april2014}
 If a normal compact operator $X\in \Ic\subsetneqq\Bc(\Hc)$ with spectral multiplicities one is 
$U_{\Ic}(\Hc)$-diagonalizable  to a diagonal operator $D$ 
with respect to the fixed  basis $\bb=\{b_n\}_{n\ge 1}$, 
 then $D$ is unique up to a finite permutation. 
 Furthermore, $X$ can be $U_{\Ic}(\Hc)$-diagonalized to every finite permutation of $D$.
\end{proposition}

\begin{proof}
Suppose $W_1$ and $W_2$ are two unitary operators in $U_{\Ic}(\Hc)$ with 
$W_1XW_1^{*} = D_1$ and $W_2XW_2^{*} = D_2$. 
This implies that $D_1 = \diag (d_n)$ and $D_2$ have the same spectrum.
Hence $D_2= V_{\sigma}D_1V_\sigma^{-1} = V_{\sigma}D_1V_{\sigma^{-1}}$ 
for some permutation $\sigma \colon\NN\to\NN$.
But then $W_2XW_2^{*} =  V_{\sigma}D_1V_{\sigma^{-1}}$, 
hence $XW_2^{*} V_{\sigma}= W_2^{*}V_{\sigma}D_1$.  
Applying this to $b_n$ we obtain 
$XW_2^{*} V_{\sigma}b_n= W_2^{*}V_{\sigma}D_1b_n$, 
so 
$$XW_2^{*}b_{\sigma^{-1}(n)} = d_nW_2^{*}b_{\sigma^{-1}(n)}.$$ 
Since $X$ has distinct eigenvalues, $W_2^{*}b_{\sigma^{-1}(n)} = \alpha_nb_n$ 
for some $\vert\alpha_n\vert = 1$ for all $n$. 
But $W_2^{*} = \textbf 1 + K$, so one obtains 
$(\1 + K)b_{\sigma^{-1}  (n)} = \alpha_nb_n$ and then 
$$\Vert Kb_{\sigma^{-1}  (n)}\Vert =  \Vert\alpha_nb_n - b_{\sigma^{-1}  (n)}\Vert.$$ 
Since $K$ is a compact operator, $\Vert Kb_{\sigma^{-1}  (n)}\Vert  \rightarrow 0$ which implies that 
$\sigma^{-1} (n) = n$ eventually.
Therefore $\sigma $ must be a finite permutation. 
This observation together with $D_2 = V_{\sigma}D_1V_{\sigma^{-1}}$ proves the claim. 
Although not used here, it is clear also that $\alpha_n \rightarrow 1$. 

To prove the last assertion of Proposition~\ref{7april2014}, 
notice that every finite permutation is a unitary of the form $\1 + F$ with $F \in \Fc(\Hc)$. 
(This is because for finite permutations $\sigma$, $V_\sigma=\1+(V_\sigma-\1)$ and since 
$V_\sigma\bb_n=\1\bb_n=\bb_n$ eventually, $V_\sigma-\1=F$ is finite rank.) 
Therefore, if $X$ is $U_{\Ic}(\Hc)$-diagonalizable to a diagonal operator $D$, that is, $W_1XW_1^{*} = D$ 
for some $W_1=\1+K_1\in U_{\Ic}(\Hc)$,  
then $X$ is also $U_{\Ic}(\Hc)$-diagonalizable to $D'$, any finite permutation of $D$.   
Indeed, $V_{\sigma^{-1}}W_1XW_1^{*}V_{\sigma} = D'$ 
for some finite permutation $\sigma$, 
and so 
$V_\sigma=\1+F$ and 
$$V_{\sigma^{-1}}W_1=(\1+F^*)(\1+K_1)=\1+F^*+K_1+F^*K_1$$
is a unitary of the form $\1 + K$ for $K \in \Ic$. 
\end{proof}

Note $W_1$ and $W_2$ are related in the above setting.  
In fact, if $W_1XW_1^*=D_1$ and $W_2XW_2^*=D_2$ are a permutation of each other 
(not necessarily finite permutation), 
without attempting to obtain an explicit form for the unitaries $W_1$ and $W_2$, 
we can however obtain an explicit form for $W_1W_2^{*}$, that is, 
information on how $W_1$ and $W_2$ relate: 
\\
Since $W_2XW_2^{*} =  V_{\sigma}D_1V_{\sigma^{-1}}$,
we have $X = W_2^{*}V_{\sigma}D_1V_{\sigma^{-1}}W_2
= W_1^{*}D_1W_1$. 
Therefore $W_2W_1^{*}D_1W_1W_2^{*} = V_{\sigma}D_1V_{\sigma^{-1}}$. 
If $U:=W_2W_1^{*}$  
then 
$UD_1U^{*} = V_{\sigma}D_1V_{\sigma^{-1}}$, 
and hence 
$$D_1(U^{*}V_{\sigma}) = (U^{*}V_{\sigma})D_1.$$ 
The multiplicities one condition on $D_1$ 
then implies that the unitary operator $U^{*}V_{\sigma}$ is a diagonal operator, 
say $D_{u}\in\Dc$. 
That is, $W_1W_2^{*}V_{\sigma} = D_u$ 
and hence  $W_1W_2^{*} = V_{\sigma^{-1}}D_u$. 
So $W_2$ in terms of $W_1$ has the form $W_2=D_u^*V_\sigma W_1$ 
(even if, as in the case of infinite $\sigma$, neither $W_1$ nor $W_2$ are in $U_{\Ic}(\Hc)$). 
In the proof of Proposition~\ref{7april2014} $\sigma$ is in addition a finite permutation.

\textit{On the structure of $\Dc_{\Ic}$ and $\Dc_{\Ic}^{\sa}$}.
The geometric shape of the set $\Dc_{\Ic}^{\sa}$ from \eqref{Dsa} is not clear in general. 
For instance Remark~\ref{WvN} shows that it need not be a real linear subspace of $\Ic^{\sa}$.
In the case when $\Ic\ne\Ic^2$, some information on the shape and size of $\Dc_{\Ic}$ and $\Dc_{\Ic}^{\sa}$
is discussed below in Remark~\ref{WvN}--Question~\ref{second_quest}, Corollary~\ref{conj}, 
and Remark~\ref{info}.

\begin{remark}\label{WvN}
\normalfont
When $\Ic=\Bc(\Hc)$ (so  $U_{\Ic}(\Hc)=U(\Hc)$), 
then $\Dc_{\Ic}$ is the set of all normal operators in $\Bc(\Hc)$
with pure point spectrum 
for which $\Hc$ decomposes into the orthogonal direct sum of the eigenspaces 
of such an operator. 
Hence $\Dc_{\Ic}^{\sa}\subsetneqq\Ic^{\sa}$ 
because of the existence of self-adjoint operators without eigenvalues.

It is instructive to note that usually (i.e., when $\Ic\ne\Bc(\Hc)$), 
the basis $\bb=\{\bb_n\}_{n\ge1}$ is tied to $\Dc$ and hence also to $\Dc_{\Ic}$. 
However, when $\Ic=\Bc(\Hc)$, all bases yield the same $\Dc_{\Ic}$. 

Moreover, $\Dc_{\Ic}^{\sa}$ is not additive. 
Indeed, the Weyl-von Neumann theorem implies that if $A=A^*\in\Ic=\Bc(\Hc)$, 
then there exist $A_1\in\Dc_{\Ic}^{\sa}$ and
$A_2=A_2^*\in\Kc(\Hc)$ for which $A=A_1+A_2$
(see for instance~\cite[Sect. II.4]{Da96}, or \cite{Ku58}
for a generalization involving symmetrically normed ideals).
Furthermore, if the self-adjoint operator $A$ does not have pure point spectrum,
then we obtain $A_1,A_2\in\Dc_{\Ic}^{\sa}$ and $A_1+A_2=A\not\in\Dc_{\Ic}^{\sa}$.
This shows that $\Dc_{\Ic}^{\sa}$ fails to be a real linear subspace of $\Ic^{\sa}$ if $\Ic=\Bc(\Hc)$.
\qed
\end{remark}

Now the following question naturally arises.

\begin{question}\label{second_quest}
\normalfont
Is it true that the final part of Remark~\ref{WvN} can be generalized 
in the sense that for every nonzero ideal $\Ic$ in $\Bc(\Hc)$, 
the set $\Dc_{\Ic}^{\sa}$  is a proper subset of $\Ic^{\sa}$ 
and beyond Remark~\ref{WvN}: does it real linearly spans $\Ic^{\sa}$? 
\qed
\end{question}

We do not know the answer to the second part of this question, that is, whether 
$\Ic^{\sa}$ is the real linear span of $\Dc_{\Ic}^{\sa}$. 
As regards the part of the above question concerning the proper inclusion $\Dc_{\Ic}^{\sa}\subsetneqq\Ic^{\sa}$, 
it is answered in the affirmative in Corollary~\ref{conj} below 
whose proof is based on the following Proposition~\ref{bases11}--Corollary~\ref{bases2}.

\begin{notation}
\normalfont
$\Dc(\vv)$ denotes the set of operators in $\Bc(\Hc)$
that are diagonal with respect to the orthonormal basis of $\Hc$, $\vv=\{\vv_n\}_{n\ge 1}$. 
\end{notation}

\begin{proposition}\label{bases11}
Let $\vv=\{\vv_n\}_{n\ge 1}$ and $\ww=\{w_n\}_{n\ge 1}$ be orthonormal bases in $\Hc$.  
Suppose the distances between the parts of the 1-dimensional subspaces spanned by $v_m$ and by $w_n$ that lie on the surface of the unit ball are uniformly bounded away from 0. 
 That is,  for some $\delta > 0$, 
$\Vert\beta w_n-\alpha v_m\Vert \ge \delta$ for all $n,m \ge 1$ and 
$\vert\beta\vert=\vert\alpha\vert=1$, or equivalently and more simply, 
$\Vert w_n-\alpha v_m\Vert \ge \delta$ for all $n,m \ge 1$ and $\vert\alpha\vert=1$.

Then for every operator $X\in\Dc(\vv)$ with spectral multiplicities one
and  every $W\in U_{\Kc(\Hc)}$ one has $WXW^{-1}\not\in\Dc(\ww)$. 
Consequently, the $U_{\Kc(\Hc)}$-orbits of the multiplicities one operators of $\Dc(\vv)$ and $\Dc(\ww)$ are disjoint. 
\end{proposition}

\begin{proof}
In fact, if $Xv_n = \lambda_n v_n$ for all $n\ge1$ 
and if $W\in U_{\Kc(\Hc)}$ and $WXW^{-1}$ were a diagonal operator with
respect to the basis $\{w_n\}_{n\ge1}$,
then the 1-dimensional subspace spanned by the vectors of the latter basis
are precisely the eigenspaces of $X$ since every eigenvalue of $X$ has multiplicity one 
and so likewise any operator in its unitary orbit has the same eigenvalues also with multiplicities one 
(i.e., these eigenvalues are distinct).
Therefore if 
$WXW^{-1} w_n = \lambda'_n w_n$  then 
 $X W^{-1} w_n = \lambda'_n W^{-1}w_n$ and where by the distinctness of the eigenvalues and their all having multiplicity one,
we have a bijection $\sigma$ and scalars  $\vert\alpha_n\vert = 1$ for which 
$\lambda'_n = \lambda_{\sigma(n)}$ and $W^{-1} w_n = \alpha_n v_{\sigma(n)}$ for all $n \ge 1$. 

Moreover, since $W^{-1}\in U_{\Kc(\Hc)}$, one has $W^{-1}=\1+K$ for some $K\in\Kc(\Hc)$.
Then for every $n\ge 1$, $Kw_n=-w_n+\alpha_n \vv_{\sigma(n)}$,
hence $\Vert Kw_n\Vert\ge\delta$.
However, since the sequence $\{w_n\}_{n\ge 1}$ is orthonormal, so weakly convergent to $0$ in~$\Hc$, 
and $K$ is a compact operator,
we have $\lim\limits_{n\to\infty}\Vert Kw_n\Vert=0$,
which contradicts being bounded below.
This negates our initial assumption that the operator $WXW^{-1}$ is diagonal with respect
to the basis $\{w_n\}_{n\ge 1}$, consequently $WXW^{-1}\not\in\Dc$. 
\end{proof}

The following corollary is a reframing of Proposition~\ref{bases11} in a slightly more complicated form, 
but in a more useful form for our approach. 

\begin{corollary}\label{bases1}
Let $\vv=\{\vv_n\}_{n\ge 1}$ and $\ww=\{w_n\}_{n\ge 1}$ be orthonormal bases in $\Hc$. 
Assume there exists $\delta>0$ such that 
for arbitrary scalar sequences $\langle\alpha_n\rangle_{n=1}^\infty$ and $\langle\beta_n\rangle_{n=1}^\infty$
in the unit circle in $\CC$ one has $\Vert \beta_n w_n-\alpha_m \vv_m\Vert\ge \delta$
for all $n,m\ge 1$.

Then for every operator $X\in\Dc(\vv)$ with spectral multiplicities one
and  every $W\in U_{\Kc(\Hc)}$ one has $WXW^{-1}\not\in\Dc(\ww)$.
\end{corollary}

\begin{proof}
Use Proposition~\ref{bases11}.
\end{proof}

\begin{proposition}\label{bases22}
Fix any orthonormal basis $\bb=\{b_n\}_{n\ge 1}$ of $\Hc$. 
There exists a family of vectors $\{\vv^{(t)}_n\mid n\in\NN,\, 0\le t<\pi/2\}$
with the following properties:
\begin{enumerate}[(i)]
\item\label{bases22_item1}
For every $t\in[0,\pi/2)$ the subfamily $\vv^{(t)}:=\{\vv^{(t)}_n\}_{n\ge 1}$ is an orthonormal basis in $\Hc$ 
 with $\vv^{(0)}=\bb$.
\item\label{bases22_item2}
If $t\ne s\in[0,\pi/2)$, then there exists $\delta>0$ with   
$\Vert v^{(t)}_n - \alpha v^{(s)}_m\Vert \ge \delta$
for all $n,m\ge 1$ and $\alpha\in\CC$ with $\vert\alpha\vert=1$.
\end{enumerate}
\end{proposition}

\begin{proof} 
\eqref{bases22_item1} 
For every $t\in[0,\pi/2)$ and $\bb=\{b_n\}_{n\ge 1}$ 
we define the orthonormal basis $\vv^{(t)}:=\{\vv^{(t)}_n\}_{n\ge 1}$ in $\Hc$ by
$$\vv^{(t)}_{2r-1}=(\cos t)\bb_{2r-1}+(\sin t)\bb_{2r}\quad \text{ and }\quad\vv^{(t)}_{2r}=-(\sin t)\bb_{2r-1}+(\cos t)\bb_{2r}
\text{ for }r\ge 1.$$
Hence 
we perform a rotation of angle $t\in[0,\pi/2)$ in the plane $\RR\cdot\bb_{2r-1}+\RR\cdot\bb_{2r}$
in order to obtain the orthonormal set  $\{\vv^{(t)}_{2r-1},\vv^{(t)}_{2r}\}$ from $\{\bb_{2r-1},\bb_{2r}\}$.
It is then clear that for each $t\in[0,\pi/2)$ the family
$\vv^{(t)}=\{\vv^{(t)}_n\}_{n\ge 1}$ is an orthonormal basis in $\Hc$ 
and that $\vv^{(0)}=\bb$.

\eqref{bases22_item2} Let $0\le s<t<\pi/2$
and 
$\alpha\in\CC$ be arbitrary with $\vert\alpha\vert=1$.
If for some $r\ge 1$ we have
$n\in\{2r-1,2r\}$ and $m\not\in\{2r-1,2r\}$, then $\vv_n^{(t)}\perp \vv_m^{(s)}$ hence 
$\Vert \alpha \vv_n^{(t)}-\vv_m^{(s)}\Vert=\sqrt{2}$.
Moreover, if  $n,m\in\{2r-1,2r\}$ then
$$
\Vert \alpha \vv_n^{(t)}- \vv_m^{(s)}\Vert^2
=2-2\text{Re}\,(\alpha(\vv_n^{(t)},\vv_m^{(s)})) 
\ge 2(1-\vert\text{Re}\,(\vv_n^{(t)},\vv_m^{(s)})\vert).
$$
On the other 
hand, since $s<t$, it follows that the orthonormal set  $\{\vv^{(t)}_{2r-1},\vv^{(t)}_{2r}\}$
is obtained from the orthonormal set  $\{\vv^{(s)}_{2r-1},\vv^{(s)}_{2r}\}$
by performing a rotation of angle $t-s$ in the plane $\RR\cdot\bb_{2r-1}+\RR\cdot\bb_{2r}$,
hence we have
$$\begin{aligned}
\vv^{(t)}_{2r-1}&=(\cos (t-s))\vv^{(s)}_{2r-1}+(\sin (t-s))\vv^{(s)}_{2r}, \\
\vv^{(t)}_{2r}&=-(\sin (t-s))\vv^{(s)}_{2r-1}+(\cos (t-s))\vv^{(s)}_{2r}.
  \end{aligned}
$$
It follows by these formulas that if  $n,m\in\{2r-1,2r\}$ then
$$\vert (\vv_n^{(t)},\vv_m^{(s)})\vert\le\max\{\cos(t-s),\sin(t-s)\}.$$
Therefore
$$\Vert \alpha \vv_n^{(t)}- \vv_m^{(s)}\Vert^2\ge 2(1-\max\{\cos(t-s),\sin(t-s)\})$$
and this lower bound is independent of $r\ge 1$ and strictly positive (as $0<t-s<\pi/2$).
Thus we see that condition~\eqref{bases2_item2} is satisfied, completing the proof.
\end{proof}

The next corollary provides uncountably many orthogonal bases in $\Hc$
that pairwise satisfy the requirements of Corollary~\ref{bases1}.
In the proof of Corollary~\ref{conj} we will actually need only one such pair. 
However the full power of the following corollary will be needed later 
in order to prove that there exist
infinitely many conjugacy classes of Cartan subalgebras of a nonzero operator ideal.

\begin{corollary}\label{bases2}
Fix any orthonormal basis $\bb=\{b_n\}_{n\ge 1}$ of $\Hc$. 
There exists a family of vectors $\{\vv^{(t)}_n\mid n\in\NN,\, 0\le t<\pi/2\}$
with the following properties:
\begin{enumerate}[(i)]
\item\label{bases2_item1}
For every $t\in[0,\pi/2)$ the subfamily $\vv^{(t)}:=\{\vv^{(t)}_n\}_{n\ge 1}$ is an orthonormal basis in $\Hc$ 
with $\vv^{(0)}=\bb$.
\item\label{bases2_item2}
If $t\ne s\in[0,\pi/2)$, then there exists $\delta>0$ with the property that
for arbitrary scalar sequences $\langle\alpha_n\rangle_{n=1}^\infty$ and $\langle\beta_n\rangle_{n=1}^\infty$
in the unit circle in $\CC$ one has $\Vert \beta_n \vv_n^{(t)}-\alpha_m \vv_m^{(s)}\Vert\ge \delta$
for all $n,m\ge 1$.
\end{enumerate}
\end{corollary}

\begin{proof}
Use Proposition~\ref{bases22}. 
\end{proof}

\begin{corollary}\label{conj}
If $\Ic\supsetneqq\Fc(\Hc)$ is an operator ideal in $\Bc(\Hc)$, 
then $\Dc_{\Ic}^{\sa}\subsetneqq\Ic^{\sa}$.
\end{corollary}

\begin{proof}
The case $\Ic=\Bc(\Hc)$ was already settled in Remark~\ref{WvN},
so we may assume $\Fc(\Hc)\subsetneqq\Ic\subsetneqq\Bc(\Hc)$. 
Pick any $t_0\in(0,\pi/2)$ and consider the
corresponding orthonormal basis $v:=\vv^{(t_0)}=\{\vv^{(t_0)}_n\}_{n\ge 1}\ne\bb$
provided by Corollary~\ref{bases2}\eqref{bases2_item1}. 
Then choose any self-adjoint operator $X\in\Dc(v)\cap\Ic$ 
and with spectral multiplicities one. 
Then using Proposition~\ref{bases11}, 
one obtains $WXW^{-1}\not\in\Dc$ for all $W\in U_{\Kc(\Hc)}$. 
So, also for all $W\in U_{\Ic}(\Hc)$ in particular. 
The latter means $X\not\in\Dc_{\Ic}^{\sa}$.  
\end{proof}

Observe choosing $X\ge 0$ in the above proof yields the stronger 
$\Dc_{\Ic}^{\sa}\cap\Ic^{+}\subsetneqq\Ic^{\sa}$. 

Corollary~\ref{conj} proves there are self-adjoint operators in $\Ic$ that are not $U_{\Ic}(\Hc)$-diagonalizable. 
Hence the following. 

\subsection*{A necessary condition for $U_{\Ic}(\Hc)$-diagonalization} \hfill 
\\
For the following theorem we recall that for any operator ideal $\Ic$ in $\Bc(\Hc)$,
the operator ideal $\Ic^2$ is defined as the linear span of the operator set $\{TS\mid T,S\in\Ic\}$ 
and so one has $\Ic^2\subseteq\Ic$. 

\begin{theorem}\label{July6}
Let $\Ic$ be any operator ideal in $\Bc(\Hc)$. 
If an operator $X\in\Ic$ is $U_{\Ic}(\Hc)$-diagonalizable, that is, $WXW^{-1}=D$ for some operators $W\in U_{\Ic}(\Hc)$ and $D\in\Dc$, 
then $X$ is a normal operator and $X-D\in\Ic^2$.  
\end{theorem}

\begin{proof}
Note that $X=W^{-1}DW$. 
Moreover, since $W\in U_{\Ic}(\Hc)$, we have $W=\1+K$ for some $K\in\Ic$, 
and hence $W^{-1}=W^*=\1+K^*$.
Therefore
$$\begin{aligned}
X
&=W^{-1}DW \\
&=(\1+K^*)D(\1+K) \\
&=D+K^*D+DK+K^*DK.
\end{aligned}$$
Since $D=WXW^{-1}\in W\Ic W^{-1}=\Ic$ and $K\in\Ic$,
one obtains $X-D\in \Ic^2$.
\end{proof}

\begin{remark}
\normalfont
There is a little more to Theorem~\ref{July6} than meets the eye. 
Reframing it creates a context for a possible converse as follows. 

Suppose $X\in\Ic$ is normal so that $WXW^{-1}=D$ for some $W\in U(\Hc)$ and $D\in\Dc$.  
Then $X-D\in\Ic^2$. 
The converse question is: 
Does $X-D\in\Ic^2$ imply $W$ can be chosen in $U_{\Ic}(\Hc)$? 
The answer to this is no. 

The reason for this is that, fixing an operator~$X\in\Dc(\vv)\cap\Ic^2$, 
$\vv\ne\bb$, using Proposition~\ref{bases11}, 
$WXW^{-1}\not\in\Dc$ for all $W\in U_{\Kc(\Hc)}$  
(hence also for all $W\in U_{\Ic}(\Hc)$). 
But then $X,D\in\Ic^2$, so $X-D\in\Ic^2$. 
But $WXW^{-1}\not\in\Dc$ hence $WXW^{-1}\ne D$ for any $W\in U_{\Ic}(\Hc)$ in particular. 
\end{remark}

\subsection*{A sufficient condition for $U_{\Sg_2(\Hc)}$-diagonalizability}\hfill 
\begin{remark}\label{suffcond}
\normalfont
For the sake of completeness, let us mention that in the case when $\Ic=\Sg_2(\Hc)$ is the Hilbert-Schmidt ideal,
some sufficient conditions
for $U_{\Ic}(\Hc)$-diagonalizability were provided in \cite[Th. 1]{Hi85} 
as follows.

Let $X\in\Ic^{\sa}$ and denote $x_{ij}=(X\bb_j, \bb_i)$ for all $i,j\ge 1$. 
If there exist $\rho,s\in\RR$ where $0<\rho<1$ and $0<s\le 3(1-\rho)/100$ such that
$$(\forall j\ge1)\quad \vert x_{j+1,j+1}\vert\le\rho\vert x_{jj}\vert$$
and
$$(\forall i,j\ge 1,\ i\ne j)\quad \vert x_{ij}\vert^2\le\frac{s^2}{(ij)^2}\cdot\vert x_{ii}x_{jj}\vert $$
then $WXW^{-1}\in\Dc$ for some $W\in U_{\Ic}(\Hc)$.

The above sufficient conditions for $U_{\Ic}(\Hc)$-diagonalizability 
are by no means necessary. 
More precisely, by using Corollary~\ref{conj} for $\Ic=\Sg_2(\Hc)$, 
one can find $X\in\Ic^{\sa}\setminus\Dc_{\Ic}^{\sa}$, 
and for such a self-adjoint operator $X$ there exist no $\rho,s\in\RR$ satisfying the above conditions, 
since $X$ fails to be $U_{\Ic}(\Hc)$-diagonalizable. 
\end{remark}
  
\subsection*{More on the size of $\Dc_{\Ic}$ under a certain homomorphism of triple systems}\hfill 

In the case when $\Ic^2\ne \Ic$, Proposition~\ref{July6_cor}\eqref{July6_cor_item3} below is related to
the part of Question~\ref{second_quest} on $\Dc_{\Ic}^{\sa}\subsetneqq\Ic^{\sa}$, 
in as much as it provides some information on the size of the set $\Dc_{\Ic}$. 

\begin{proposition}\label{July6_cor}
There exists an injective linear mapping
$\Pi\colon\Bc(\Hc)\to\Bc(\Hc)$ with the properties:
\begin{enumerate}[(i)] 
\item\label{July6_cor_item0}
$\Pi$ is a homomorphism of triple systems, that is,
for all $R,S,T\in\Bc(\Hc)$ we have $\Pi(RST)=\Pi(R)\Pi(S)\Pi(T)$.
\item\label{July6_cor_item0.5}
If $S,T\in\Bc(\Hc)$, then $ST=TS$ if and only if $\Pi(S)\Pi(T)=\Pi(T)\Pi(S)$.
\item\label{July6_cor_item1}
For every $T\in\Bc(\Hc)$, $\Pi(T^*)=\Pi(T)^*$.
\item\label{July6_cor_item2}
$\Bc(\Hc)^{+}\cap\Ran\Pi=\{0\}$, i.e., $\Ran\Pi$ contains no nonzero positive operators.
\item\label{July6_cor_item2.5}
If $\Ic_1$ and $\Ic_2$ are any operator ideals in $\Bc(\Hc)$,
then $\Ic_1\subseteq\Ic_2$ if and only if $\Pi(\Ic_1)\subseteq\Pi(\Ic_2)$.
\item\label{July6_cor_item3}
For every operator ideal $\Ic$ in $\Bc(\Hc)$ we have $\Pi(\Ic)\subseteq\Ic$, $\Pi^{-1}(\Ic)\subseteq\Ic$, and
$\Pi^{-1}(\Dc_{\Ic})\subseteq\Ic^2$, 
where for any $\Pi^{-1}(\Ac):=\{T\in\Bc(\Hc)\mid \Pi(T)\in\Ac\}$ for any $\Ac\subseteq\Bc(\Hc)$.
\end{enumerate}
\end{proposition}

\begin{proof}
Express the vectors in $\widetilde{\Hc}:=\Hc\oplus\Hc$ as two vertically listed column vectors and the operators
in $\Bc(\widetilde{\Hc})$ as $2\times 2$ operator matrices with entries in $\Bc(\Hc)$.
Define in terms of the fixed basis $\bb=\{b_n\}_{n\ge 1}$ of~$\Hc$ 
the unitary operator $\Delta\colon \Hc\to\widetilde{\Hc}$
as
$$\Delta(\bb_{2n-1})=
\begin{pmatrix}\bb_n \\ 0\end{pmatrix}
\text{ and }
\Delta(\bb_{2n})=
\begin{pmatrix} 0 \\ \bb_n \end{pmatrix}
\text{ for all } n\ge1.$$
Define the operator ideal $\widetilde{\Ic}:=\Delta\Ic\Delta^{-1}$ in $\Bc(\widetilde{\Hc})$.
Recall that unitary operators between Hilbert spaces in this way preserve $s$-numbers and hence preserve operator ideals,
that is, $X\mapsto \Delta X\Delta^{-1}$ maps the class of operator ideals in $\Bc(\Hc)$
one-to-one inclusion preserving onto the class of operator ideals in $\Bc(\widetilde{\Hc})$.
We will check that the injective linear mapping
$$\Pi\colon\Bc(\Hc)\to\Bc(\Hc),\quad
T\mapsto \Delta^{-1}\begin{pmatrix} 0 & T\\
T & 0\end{pmatrix}\Delta $$
satisfies the conditions required in the statement.
Conditions~\eqref{July6_cor_item0}--\eqref{July6_cor_item1} and \eqref{July6_cor_item2.5} are straightforward 
(for \eqref{July6_cor_item1} note that $\Delta^*=\Delta^{-1}$ since $\Delta$ is unitary). 

Condition~\eqref{July6_cor_item2}
is a direct consequence of \cite[Lemma 6.3(i)]{BRT07} 
or can be proved directly by noting that, since $\Delta$ is unitary, 
it suffices to check that for every nonzero $T\in\Bc(\Hc)$ 
the operator $\begin{pmatrix}0 & T \\ T & 0\end{pmatrix}$ fails to be positive. 
And this can be 
 obtained by using the special unitary 
$U = \frac{1}{\sqrt 2} \begin{pmatrix} \hfil \1 & -\1 \\ \hfil \1 & \hfill \1\end{pmatrix}$ 
for which 
$U\begin{pmatrix}\hfill 0 & T \\ T & 0\end{pmatrix} U^* = 
\begin{pmatrix} -T & 0 \\ \hfill 0 & T \end{pmatrix}$,  
and it is clear that the latter $2\times 2$ block matrix is neither positive nor negative if $T\ne0$.

To check that Condition~\eqref{July6_cor_item3} is also satisfied,
note that $\Pi(\Ic)\subseteq\Ic$ and $\Pi^{-1}(\Ic)\subseteq\Ic$ follow directly by the definition of $\Pi$.
Next, for $A\in\Pi^{-1}(\Dc_{\Ic})$, one needs to show $A\in\Ic^2$.
Define
$$\widetilde{X}
:=\begin{pmatrix}
0 & A \\
A & 0
\end{pmatrix}\in\Bc(\widetilde{\Hc})$$
and set $X:=\Delta^{-1}\widetilde{X}\Delta$.
Since $X=\Pi(A)\in\Dc_{\Ic}$, one has $W=\1+K\in U_{\Ic}(\Hc)$ for which 
$(I+K)X(I+K^*) = D'\in\Dc$ and $K \in\Ic$ implying
\begin{equation}\label{August28}
(\Delta(\1+K)\Delta^{-1})
\underbrace{(\Delta X\Delta^{-1})}_{=\widetilde{X}}
(\Delta(\1+K^{*})\Delta^{-1}) = \Delta D'\Delta^{-1}.
\end{equation} 
Then set $U_{\widetilde{\Ic}}(\widetilde{\Hc}):= U(\widetilde{\Hc}) \bigcap (\1 + \widetilde{\Ic})$
where $U(\widetilde{H})$ denotes the group of unitary operators on $\widetilde{\Hc}$.
Notice that 
$$\begin{aligned}
\Delta(\1+K)\Delta^{-1} & = \1 + \Delta K \Delta^{-1}
\in U_{\widetilde{\Ic}}(\tilde H),\\
\widetilde{X}&=\Delta X\Delta^{-1}
\in \widetilde{\Ic}, \\
\Delta(\1+K^{*})\Delta^{-1}&= \1 + \Delta K^{*}\Delta^{-1}
  \in U_{\widetilde{\Ic}}(\widetilde{\Hc}).
\end{aligned}$$
Also notice that $\Delta D'\Delta^{-1}$ is a diagonal operator in $B(\widetilde{\Hc})$ 
relative to the basis 
$\Bigl\{\begin{pmatrix} b_n\\ 0\end{pmatrix},\begin{pmatrix} 0\\ b_n\end{pmatrix}\mid n\ge 1\Bigr\}$
because of the particular way the operator $\Delta$ was earlier defined.
So from Equation~\eqref{August28} one obtains for $\widetilde{X}\in\widetilde{\Ic}$,
a unitary operator $\widetilde{W} := \1 + \Delta K \Delta^{-1}\in \mathcal U_{\widetilde{\Ic}}(\widetilde{\Hc})$ for which $\widetilde{W}\widetilde{X} \widetilde{W}^{-1} = \Delta D'\Delta^{-1}=: D''$.

Since $\Pi(A)$ is a normal compact operator, $\widetilde{X}$ is a normal compact operator.
So applying Theorem \ref{July6} in the $B(\widetilde{\Hc})$ setting to
$\widetilde{W}\widetilde{X} \widetilde{W}^{-1}$ relative to the basis 
$\Bigl\{\begin{pmatrix} b_n\\ 0\end{pmatrix},\begin{pmatrix} 0\\ b_n\end{pmatrix}\mid n\ge 1\Bigr\}$, 
one obtains $\widetilde{X} - D'' \in \widetilde{\Ic}^{2}$.
Diagonals in the $\Bc(\widetilde{\Hc})$ setting are regarded as direct sums of diagonals in $\Bc(\Hc)$. 
Let us write 
$$
D''
= \begin{pmatrix}
D_{1}'' &  0            \\
0      &  D_{2}''
\end{pmatrix}, 
\text{ hence }
\widetilde{X} - D''
= 
\begin{pmatrix}
-D_{1}'' & \hfil  A  \\
      A & -D_{2}''
\end{pmatrix}
\in \tilde{\mathcal I}^{2}.
$$
For $\widetilde{S}:=\begin{pmatrix} \1 & \hfill 0 \\ 0 & -\1 \end{pmatrix}\in U(\widetilde{\Hc})$ one has
$$\begin{pmatrix}
0 &  2A            \\
2A&0                 
\end{pmatrix}
=
\begin{pmatrix}
-D^{''}_{1} &  A            \\
 A&-D^{''}_2               
\end{pmatrix}
- \widetilde{S}^*\begin{pmatrix}
-D_{1}'' &   A  \\
      A & -D_{2}''
\end{pmatrix}
\widetilde{S}
= 
\widetilde{X} - D''-\widetilde{S}^*(\widetilde{X} - D'')\widetilde{S}
\in\widetilde{\Ic}^{2} 
$$
and therefore $A \in\Ic^{2}$, which completes the proof.
\end{proof}

\begin{remark}\label{info}
\normalfont 
As mentioned immediately prior to Proposition~\ref{July6_cor}, 
in the case when $\Ic^2\ne \Ic$, Proposition~\ref{July6_cor}\eqref{July6_cor_item3} is related to
the part of Question~\ref{second_quest} on whether or not $\Dc_{\Ic}^{\sa}\subsetneqq\Ic^{\sa}$, 
in as much as it provides some information on the size of the set $\Dc_{\Ic}$. 
We can then begin to develop intuition on how 
$\Dc_{\Ic}^{\sa}$  
sits inside of the real linear space $\Ic^{\sa}$ 
beyond this set inequality 
in Corollary~\ref{conj}. 
As for instance how the non-additive set 
$\Dc_{\Ic}^{\sa}$ 
(non-additive at least for $\Ic=\Bc(\Hc)$, cf. Remark~\ref{WvN})
sits inside the real linear space $\Ic^{\sa}$. 

More precisely as related to $\Pi$, it follows by Proposition~\ref{July6_cor}\eqref{July6_cor_item3} 
and injectivity respectively that 
$$\Dc_{\Ic}\cap\Pi(\Ic)\subseteq\Pi(\Ic^2)\subsetneqq\Pi(\Ic).$$
Moreover, 
$$\{0\}\ne\Dc_{\Ic}^{\rm sa}\cap\Pi(\Ic^{\rm sa})\subsetneqq\Pi(\Ic^{\rm sa}).$$
Indeed, since the mapping $\Pi\colon\Ic\to\Ic$ is linear and injective 
and preserves selfadjoints (see Proposition~\ref{July6_cor}\eqref{July6_cor_item1}), 
the image of $\Ic^{\sa}$, $\Pi(\Ic^{\sa})$, 
is an infinite-dim\-ensio\-nal real linear subspace of $\Ic^{\sa}$ 
with the property that the real linear subspace spanned by $\Dc_{\Ic}^{\sa}\cap\Pi(\Ic^{\sa})$
is a proper real linear subspace of $\Pi(\Ic^{\sa})$.  
It is proper because $\Dc_{\Ic}^{\rm sa}\cap\Pi(\Ic^{\rm sa})=\Pi(\Ic^{\rm sa})$ implies by 
Proposition~\ref{July6_cor}\eqref{July6_cor_item3} that $\Ic^2\supseteq\Pi^{-1}(\Dc_{\Ic}^{\rm sa})\cap\Ic^{\rm sa}
=\Ic^{\rm sa}$ contradicting $\Ic^2\not\supset\Ic^{\rm sa}$.  
Note also that we have $\Dc_{\Ic}^{\sa}\cap\Pi(\Ic^{\sa})\ne\{0\}$ since by \eqref{df}
we have $\{0\}\ne\Fc(\Hc)^{\sa}=\Dc_{\Fc(\Hc)}^{\sa}\subseteq \Dc_{\Ic}^{\sa}$,
hence by Proposition~\ref{July6_cor}\eqref{July6_cor_item3},  
$\{0\}\subsetneqq\Pi(\Fc(\Hc)^{\sa})\subseteq\Fc(\Hc)^{\rm sa}\subseteq\Dc_{\Ic}^{\sa}\cap\Pi(\Ic^{\sa})$.

As another geometric feature,
it follows by Proposition~\ref{July6_cor}\eqref{July6_cor_item2} that the aforementioned real linear subspace $\Pi(\Ic^{\sa})$ 
meets the positive cone $\Ic^{+}$ only at its vertex~$0$.
\end{remark}

\begin{example}
\normalfont
If $0< p<\infty$ and $\Ic=\Sg_p(\Hc)$ is the Schatten $p$-ideal, then
Proposition~\ref{July6_cor} provides some nontrivial information on the class $\Dc_{\Ic}$. 
In fact, we have $\Ic^2=\Sg_{p/2}(\Hc)$, hence
$\Ic^2\ne \Ic$ and so the above Remark~\ref{info} applies.
\end{example}

\section{Conjugacy classes of Cartan subalgebras of operator ideals}

Cartan subalgebras of various types of infinite-dimensional Lie algebras
have been studied extensively  
(see for instance \cite{Sch60}, \cite{Sch61}, \cite{BP66}, \cite{dlH72}, \cite{St75}, \cite{NP03}, \cite{Re08},
and the references therein).
In this section we focus on
Cartan subalgebras of operator ideals. 
More specifically, we first wish to raise the classification problem for the conjugacy classes of these Cartan subalgebras.
(See \cite[Rem. 4.2]{BPW13b} for the case of $\Bc(\Hc)$.)     
Our main result, the case of proper ideals of $\Bc(\Hc)$, 
is recorded as Theorem~\ref{uncountable} on the uncountability of such conjugacy classes. 

\subsection*{The set of conjugacy classes of Cartan subalgebras}
The relevant facts concerning
this issue in the case of the operator ideal $\Ic=\Bc(\Hc)$ can be found in \cite[Rem. 4.2]{BPW13b}.  
For the main general Theorem~\ref{uncountable} below for arbitrary proper ideals, 
it is convenient to introduce the following.

\begin{definition}\label{decomp}
\normalfont
For every proper operator ideal $\Ic$ in $\Bc(\Hc)$ denote by $\Cartan(\Ic)$
the set of all \textit{Cartan subalgebras} of $\Ic$, 
that is, the set of all maximal abelian self-adjoint subalgebras
of the associative $*$-algebra~$\Ic$.

Furthermore, denote by $\Decomp(\Hc)$ the set of
all sets of one-dimensional, mutually orthogonal subspaces of the Hilbert space $\Hc$ whose linear span is dense in $\Hc$. 
(I.e., all one-dimensional subspace orthogonal decompositions of~$\Hc$.)  
That is, for every $\Sc\in\Decomp(\Hc)$ we have the orthogonal direct sum decomposition $\Hc=\bigoplus\limits_{\Vc\in\Sc}\Vc$
with one-dimensional summands~$\Vc$.
\end{definition}

\begin{remark}\label{roots}
\normalfont
We should point out that this Definition~\ref{decomp} terminology 
agrees with that introduced in \cite[Sect. I.3]{dlH72} 
for the Schatten ideals, the finite-rank operator ideal, and for $\Bc(\Hc)$,  
where the Cartan subalgebras were actually defined as 
maximal abelian self-adjoint subalgebras of a complex \textit{Lie algebra}~$\gg$ 
endowed with an antilinear involution. 
If $\gg$ is finite-dimensional, it can be recovered from any of its Cartan subalgebras 
$\hg$ as follows.  
Denote the eigenspaces for the adjoint action of $\hg$ on $\gg$ by 
$$\gg^\alpha:=\{X\in\gg\mid(\forall H\in\hg)\quad [H,X]=\alpha(H)X\}
\text{ for }\alpha\colon\hg\to\CC\text{ linear}$$
and consider the corresponding set of roots 
$\Delta(\gg,\hg):=\{\alpha\in\hg^*\mid \gg^\alpha\ne\{0\}\}$. 
Then $\hg=\gg^0=\{X\in\gg\mid[\hg,X]=\{0\}\}$ and there exists the root space decomposition 
\begin{equation}\label{roots_eq1}
\gg=\hg\oplus\bigoplus_{\alpha\in\Delta(\gg,\hg)\setminus\{0\}}\gg^\alpha
\end{equation}
by \cite[Equations (2.16) and (2.22)]{Kn02}. 
If however $\gg$ is infinite-dimensional, then \eqref{roots_eq1} may not 
hold true, 
and additional structure on $\gg$ is necessary in order to formulate useful versions of that 
root space decomposition. 
For instance, if $\gg=\Mc$ is a $C^*$- or von Neumann algebra,  
then the regularity property of a masa $\hg=\Cc$ as expressed in \eqref{reg} 
should be viewed as a variant of root space decomposition of $\Mc$ with respect to~$\Cc$, 
in as much as in both situations the ambient algebra is recovered by using only the Cartan subalgebra under consideration. 

In the present situation of the Lie algebra defined by the operator ideal $\Ic$ with the Lie bracket $[X,Y]=XY-YX$,
a Cartan subalgebra of $\Ic$ is a linear self-adjoint subspace $\Cc$ of $\Ic$ which
is maximal under the condition that for all $X,Y\in\Cc$ we have $[X,Y]=0$.
The maximality condition requires however that $\Cc$ is actually a subalgebra of the associative algebra $\Ic$. 
That is, for all $X,Y\in\Cc$ one necessarily has $XY\in\Cc$, since otherwise one could replace $\Cc$
by the associative algebra it generates, 
and one would thus obtain a larger linear self-adjoint subspace
consisting of mutually commuting operators. 
We thus see that the elements of $\Cartan(\Ic)$ are precisely Cartan subalgebras of $\Ic$ 
as defined in our Introduction.
\end{remark}

Theorem~\ref{suppl6} below shows that the Cartan subalgebras of complete separable normed ideals
enjoy a regularity property which is quite similar to that of the Cartan subalgebras of $C^*$-algebras
(see Remark~\ref{roots} and \eqref{reg} in the Introduction). 
Its proof is based on the following simple fact   
which in particular generalizes \cite[Ch. I, Prop. 3A]{dlH72}
from the ideal of finite-rank operators to an arbitrary ideal.

\begin{proposition}\label{descr}
Let $\Ic$ be any proper operator ideal in $\Bc(\Hc)$.
The mapping
$$\Psi\colon\Decomp(\Hc)\to\Cartan(\Ic), \quad \Sc\mapsto\{T\in\Ic\mid(\forall\Vc\in\Sc)\quad T\Vc\subseteq\Vc\} $$
is one-to-one and onto.
For every $\Cc\in\Cartan(\Ic)$, the set $\Psi^{-1}(\Cc)$ is the set of 
range spaces of minimal projections of the commutative $C^*$-algebra of compact operators generated by $\Cc$.
\end{proposition}

\begin{proof}
We first check that if $\Sc\in\Decomp(\Hc)$ then $\Psi(\Sc)\in\Cartan(\Ic)$.
It is clear that $\Psi(\Sc)$ is an abelian self-adjoint subalgebra of $\Ic$,
so we still have to check that it is maximal with these properties.
To this end let $T\in\Ic$ with $[T,\Psi(\Sc)]=\{0\}$,   
and so it suffices to prove $T\in\Psi(\Sc)$. 
Indeed, for arbitrary $\Vc\in\Sc$, if we denote by $P_{\Vc}$ the orthogonal projection onto $\Vc$,
then $P_{\Vc}$ is a rank-one operator.
Moreover, since $\Sc$ is a set of mutually orthogonal subspaces of $\Hc$,
it follows that each element of $\Sc$ is invariant under $P_{\Vc}$, hence $P_{\Vc}\in\Psi(\Sc)$.
The commutator zero assumption on $T$ then implies $TP_{\Vc}=P_{\Vc}T$.
Since $\Vc\in\Sc$ is arbitrary, we obtain $T\in\Psi(\Sc)$.
We thus see that $\Psi(\Sc)$ is a maximal abelian self-adjoint subalgebra of $\Ic$,
that is, $\Psi(\Sc)\in\Cartan(\Ic)$.

Next, for arbitrary $\Cc\in\Cartan(\Ic)$, denote by $\Phi(\Cc)$
the set of all minimal projections in the commutative $C^*$-algebra of compact operators generated by $\Cc$.
Since every $C^*$-algebra of compact operators is the direct sum of full algebras of compact operators
on mutually orthogonal subspaces of $\Hc$
(see for instance \cite[Th. I.10.8]{Da96}),
it is then straightforward to prove that the dimensions of all these subspaces are~one  
by using the commutativity of $\Cc$, and hence that $\Phi(\Cc)\in\Decomp(\Hc)$.

Moreover, it is not difficult to check that the mappings $\Psi$ and $\Phi$ are inverse to each other.
\end{proof}

\begin{theorem}\label{suppl6}
If $\Ic$ is a complete separable normed ideal and $\Cc\in\Cartan(\Ic)$ with the normalizer
$U_{\Ic,\Cc}(\Hc):=\{W\in U_{\Ic}(\Hc)\mid W\Cc W^{-1}=\Cc\}$,
then the linear span of 
$$\{WD\mid W\in U_{\Ic,\Cc}(\Hc),D\in\Cc\}$$ 
is a dense linear subspace of $\Ic$.
\end{theorem}

\begin{proof} 
Use Proposition~\ref{descr} to find $\Sc=\{\Vc_n\mid n\ge 1\}\in\Decomp(\Hc)$ 
with $\Psi(\Sc)=\Cc$. 
Then for every $n\ge 1$ pick any $v_n\in\Vc_n$ with $\Vert v_n\Vert=1$. 
In this way we obtain an orthonormal basis $v:=\{v_n\}_{n\ge 1}$ of $\Hc$ 
with the property that $\Cc$ is the set of all operators in $\Ic$ 
that are diagonal with respect to the orthonormal basis~$v$. 
For every integer $n\ge 1$ let $P_n$ be the orthogonal projection
of $\Hc$ onto $\spann\{v_1,\dots,v_n\}$, hence $P_n\in\Fc(\Hc)$.
Also denote
$$\Qc:=\{WD\mid W\in U_{\Ic,\Cc}(\Hc),D\in\Cc\}\subset\Ic.$$ 
By using the classical variant of Proposition~\ref{suppl1} for $\dim\Hc<\infty$,
it is easily checked and in fact is well known that
$$P_n\Ic P_n=\spann(P_n\Qc P_n)\subseteq\spann\,\Qc.$$
(For a more general case see for instance \cite[Th. 3.1.1]{GW09}.)

On the other hand, since $\Ic$ is a complete separable normed ideal, it follows by \cite[Th. 6.3 in Ch. III]{GK69} 
that for every $T\in\Ic$ we have $\lim\limits_{n\to\infty}\Vert T-P_nTP_n\Vert_{\Ic}=0$.
Then the conclusion of the above paragraph shows that $\spann\,\Qc$ is dense in $\Ic$. 
\end{proof}

\begin{remark}[$\Ic$-equivalent bases]\label{equiv_bases}
\normalfont
Let $\Ic$ be an operator ideal in $\Bc(\Hc)$.
Proposition~\ref{descr} suggests to us an equivalence relation on the set of all orthonormal bases in $\Hc$.
Namely, if $e=\{e_n\}_{n\ge 1}$ and $f=\{f_n\}_{n\ge 1}$ are two orthonormal bases in $\Hc$,
then we say they are \textit{$\Ic$-equivalent to each other} if there exists $W\in U_{\Ic}(\Hc)$
with $We_n=f_n$ for all $n\ge 1$.
It seems an interesting problem to provide criteria, involving the vectors only, 
ensuring when two bases are equivalent in this sense,
and to study the corresponding equivalence classes of bases for particular choices of operator ideals.

In the case of the Hilbert-Schmidt ideal $\Ic=\Sg_2(\Hc)$,
it is easily checked that the bases $e$ and $f$ are $\Ic$-equivalent to each other
if and only if 
\begin{equation}\label{probl12}
\sum\limits_{n\ge 1}\Vert e_n-f_n\Vert^2<\infty.
\end{equation}  
Also recall from \cite[Probl. 12]{Ha82} that if $e=\{e_n\}_{n\ge 1}$ and $f=\{f_n\}_{n\ge 1}$ 
are any orthonormal sequences in $\Hc$ satisfying~\eqref{probl12}, 
then sequence $e$ spans $\Hc$ (and hence is an orthonormal basis) 
if and only if so does the sequence~$f$. 

The relationship between this $\Ic$-equivalence and the setting of Proposition~\ref{descr}
relies on the $U_{\Ic}(\Hc)$-equivariant correspondence
$\{e_n\}_{n\ge 1}\mapsto \{\CC e_n\}_{n\ge 1}$
from the set of all orthonormal bases in $\Hc$ onto the set $\Decomp(\Hc)$. 
(A \emph{$G$-equivariant} map $T\colon X\to Y$ is a map for which 
$T(gx)=gTx$ for all $g\in G$ and $x\in X$, for any group $G$ acting on the sets 
$X,Y$, with its corresponding group actions denoted by juxtaposition,
hence $(g,x)\mapsto gx$ and $(g,y)\mapsto gy$, respectively.)
\end{remark}

\begin{remark}
\normalfont
Let $\Ic$ be a proper operator ideal in $\Bc(\Hc)$, $\Sc\in\Decomp(\Hc)$, and define
$$E_{\Sc}\colon\Bc(\Hc)\to\Bc(\Hc),\quad E_{\Sc}(T)=\sum_{\Vc\in\Sc} P_{\Vc}TP_{\Vc}.$$
Since the map $\Psi$ in Proposition~\ref{descr} takes values in $\Cartan(\Ic)$ as indicated in its statement,  
it follows that $\Ic\cap \Ran E_{\Sc}\in\Cartan(\Ic)$.

Moreover one has the characterization  $E_{\Sc}(\Ic) = \Ic^{-{\rm am}} \cap \Ran E_{\Sc}$ where
$\Ic^{-{\rm am}}$ denotes the \textit{arithmetic mean closure} of $\Ic$,
that is, the set of all operators $T\in\Bc(\Hc)$ for which there exists $K\in\Ic$
such that for every integer $n\ge 1$ we have $s_1(T)+\cdots+s_n(T)\le s_1(K)+\cdots+s_n(K)$.
Then $\Ic^{-{\rm am}}\cap\Ran E_{\Sc}=E_{\Sc}(\Ic)$ by \cite[Cor. 4.4]{KW12}.

Consequently, \textit{if the ideal $\Ic$ is arithmetic mean closed,
that is, $\Ic^{-{\rm am}}=\Ic$, then for every $\Sc\in\Decomp(\Hc)$ we have $E_{\Sc}(\Ic)\in\Cartan(\Ic)$}.
\end{remark}

\begin{definition}\label{action}
\normalfont
The natural action of the full unitary group $U(\Hc)$ on the set $\Decomp(\Hc)$ is 
$$\alpha\colon U(\Hc)\times\Decomp(\Hc)\to\Decomp(\Hc),\quad
(W,\Sc)\mapsto\alpha_W(\Sc):=\{W(\Vc)\mid \Vc\in\Sc\}$$
and for every $\Sc\in\Decomp(\Hc)$ we define 
\begin{equation}\label{action_eq1}
U(\Hc)_{\Sc}:=\{W\in U(\Hc)\mid \alpha_W(\Sc)=\Sc\}.
\end{equation}
\indent For any proper operator ideal $\Ic$ in $\Bc(\Hc)$ 
the group action of $U(\Hc)$ on $\Cartan(\Ic)$ is: 
$$\beta\colon U(\Hc)\times \Cartan(\Ic)\to\Cartan(\Ic),\quad (W,\Cc)\mapsto\beta_W(\Cc):=W\Cc W^{-1}.$$
\end{definition}

\begin{remark}\label{equivar}
\normalfont
It is easily seen that for every proper ideal $\Ic$ in $\Bc(\Hc)$
the bijective mapping $\Psi$ in Proposition~\ref{descr} is $U(\Hc)$-equivariant,
hence the diagram
$$\begin{CD}
U(\Hc)\times \Decomp(\Hc) @>{\alpha}>> \Decomp(\Hc) \\
@V{\id_{U(\Hc)}\times\Psi}VV @VV{\Psi}V \\
U(\Hc)\times \Cartan(\Ic) @>{\beta}>> \Cartan(\Ic)
\end{CD}$$
commutes. 
This simple fact allows one to derive properties of the action $\beta$ from properties 
of the action $\alpha$. 
For instance, one can draw the following direct consequences for any Cartan subalgebra $\Cc$ of $\Ic$, 
by merely considering their corresponding properties of $\Sc:=\Psi(\Cc)\in\Decomp(\Hc)$: 
\begin{enumerate}[(i)]
\item\label{equivar_item1} 
One has $\{\beta_W(\Cc)\mid W\in U(\Hc)\}=\Cartan(\Ic)$, since clearly $\{\alpha_W(\Sc)\mid W\in U(\Hc)\}=\Decomp(\Hc)$. 
\item\label{equivar_item2} 
One has $\{W\in U(\Hc)\mid \beta_W(\Cc)=\Cc\}=U(\Hc)_{\Sc}$, using \eqref{action_eq1}.
\end{enumerate}
Moreover, if one denotes by $[\Cartan(\Ic)]$ the set of all $U_{\Ic}(\Hc)$-conjugacy classes 
of Cartan subalgebras of $\Ic$, 
and for every $\Cc\in\Cartan(\Ic)$  and $\Ac\subseteq U_{\Ic}(\Hc)$ 
one denotes $\beta_{\Ac}(\Cc):=\{\beta_V(\Cc)\mid V\in \Ac\}$, 
then for any fixed $\Cc_0\in\Cartan(\Ic)$ it is easily seen that the map 
\begin{equation}\label{equivar_eq1}
U(H)/U_{\Ic}(H)\to [\Cartan(\Ic)],\quad W U_{\Ic}(H)\mapsto \beta_W(\beta_{U_{\Ic}(\Hc)}(\Cc_0))
\end{equation}
is well-defined and bijective, hence can be regarded as a parameterization of the set $[\Cartan(\Ic)]$, 
where $W U_{\Ic}(H):=\{WV\mid V\in U_{\Ic}(H)\}$ for all $W\in U(\Hc)$. 
The above map takes values as indicated since $\beta_{U_{\Ic}(\Hc)}(\Cc_0)$ is the $U_{\Ic}(\Hc)$-conjugacy class 
of $\Cc_0$, and for every $W\in U(\Hc)$ and $V\in U_{\Ic}(\Hc)$ one has 
$\beta_W(\beta_V(\Cc_0))=\beta_{WVW^{-1}}(\beta_W(\Cc_0))$.  
But the map $V\mapsto WVW^{-1}$ is a bijection of the set $U_{\Ic}(\Hc)$ onto itself 
since $\Ic$ is an ideal of $\Bc(\Hc)$, 
hence we obtain $\beta_W(\beta_{U_{\Ic}(\Hc)}(\Cc_0))=\beta_{U_{\Ic}(\Hc)}(\Cc)$, 
where $\Cc:=\beta_W(\Cc_0)$ runs over $\Cartan(\Ic)$ when $W$ runs over $U(\Hc)$ 
(by Remark~\ref{spectral}). 
\end{remark}

\begin{theorem}\label{uncountable}
If $\Ic$ is a proper operator ideal in $\Bc(\Hc)$,
then there exist uncountably many $U_{\Ic}(\Hc)$-conjugacy classes of Cartan subalgebras of~$\Ic$.
\end{theorem}

\begin{proof}
Let $\{\vv^{(t)}\}_{0\le t<\pi/2}$ be the family of orthonormal bases in $\Hc$ provided by Corollary~\ref{bases2}.
For every $t\in[0,\pi/2)$ we have $\{\CC \vv^{(t)}_n\}_{n\ge1}\in\Decomp(\Hc)$. 
Hence by Proposition~\ref{descr} one obtains $\Cc^{(t)}:=\Psi(\{\CC \vv^{(t)}_n\}_{n\ge1})\in\Cartan(\Ic)$.
Equivalently, if we denote by $\Dc(\vv^{(t)})$ the set of all diagonal operators with respect to
the orthonormal basis $\vv^{(t)}$, then one has $\Cc^{(t)}=\Dc(\vv^{(t)})\cap\Ic$.
We will prove that the uncountable family $\{\Cc^{(t)}\}_{0\le t<\pi/2}$
is a (possibly non-exhaustive) system of representatives for 
distinct $U_{\Ic}(\Hc)$-conjugacy classes of Cartan subalgebras of $\Ic$. 

To this end let $s,t\in[0,\pi/2)$ be arbitrary with $s\ne t$.
We need to show that for every $W\in U_{\Ic}(\Hc)$ one has $\beta_W(\Cc^{(s)})\ne\Cc^{(t)}$.
Two cases may occur:

Case 1. 
If $\Fc(\Hc)\subsetneqq\Ic$, then there exists an operator $X\in\Dc(\vv^{(s)})\cap\Ic=\Cc^{(s)}$
with spectral multiplicities one.
It then follows by Corollary~\ref{bases2}\eqref{bases2_item2} along with Corollary~\ref{bases1} that
for every $W\in U_{\Ic}(\Hc)\subseteq U_{\Kc(\Hc)}$, 
$WXW^{-1}\not\in \Dc(\vv^{(t)})$, i.e., $\beta_W(\Cc^{(s)})\ne\Cc^{(t)}$.

Case 2. 
If $\Ic=\Fc(\Hc)$, assume 
$\beta_W(\Cc^{(s)})=\Cc^{(t)}$ for some $W\in U_{\Ic}(\Hc)$,
and this will lead to a contradiction.
In fact, define $X\in\Bc(\Hc)$ by $X\vv^{(t)}_n=(1/n)\vv^{(t)}_n$ for every $n\ge 1$,
and for arbitrary $k\ge 1$ define $X_k\in\Fc(\Hc)$ by $X_k\vv^{(t)}_n=(1/n)\vv^{(t)}_n$ if $1\le k\le n$
and $X_k\vv^{(t)}_n=0$ if $k>n$.
Then for $k\ge 1$ one has $X_k\in\Cc^{(t)}$,
hence $W^{-1}X_kW\in\beta_{W^{-1}}(\Cc^{(t)})=\Cc^{(s)}\subseteq\Dc(\vv^{(s)})$ by our assumption $\beta_W(\Cc^{(s)})=\Cc^{(t)}$.
Since one knows $\lim\limits_{k\to\infty}\Vert X_k-X\Vert=0$ and $\Dc(\vv^{(s)})$ is closed in $\Bc(\Hc)$
with respect to the operator norm topology, then it follows that
$W^{-1}XW\in\Dc(\vv^{(s)})$.
On the other hand $X\in \Dc(\vv^{(t)})$ is a self-adjoint operator with spectral multiplicities one
and $W\in U_{\Ic}(\Hc)\subseteq U_{\Kc(\Hc)}$,
hence it follows by Corollary~\ref{bases2}\eqref{bases2_item2} along with Corollary~\ref{bases1}
that $W^{-1}XW\not\in\Dc(\vv^{(s)})$, 
and this achieves the contradiction. 
\end{proof}

The following proposition is helpful for computing the normalizers of the Cartan subalgebras of operator ideals.
This could be proved by using Proposition~\ref{isotropy},
but we prefer to provide here an alternative method of proof, based on matrix computations.
Recall that $\Dc$ stands for the set of all diagonal operators with respect to a fixed orthonormal basis in~$\Hc$.

\begin{proposition}\label{suppl1}
If the operator
$T\in U(\Hc)$ satisfies $T(\Dc\cap\Fc(\Hc)) T^*\subseteq\Dc$,
then there exist a unique permutation $\sigma\in\SS_\infty$ and
a unique sequence $u=\langle u_n\rangle_{n=1}^\infty$ with $u_n\in\TT$ for every $n\ge 1$,
for which $T=V_\sigma\diag\, u$.

If moreover $T\in U_{\Kc(\Hc)}$, then $\sigma\in\SS_{\fin}$.
\end{proposition}

\begin{proof}
For $T=(t_{jk})_{j,k\ge 1}\in U(\Hc)$ and $D=\diag(d_1,d_2,d_3,\dots)\in \Dc\cap\Fc(\Hc)$ one has
\allowdisplaybreaks
\begin{align}
TDT^*
&= \begin{pmatrix}
t_{11}& t_{12} & t_{13}& \dots \\
t_{21}& t_{22}& t_{23}& \dots \\
t_{31}& t_{32} & t_{33}& \dots\\
\vdots & \vdots & \vdots &
  \end{pmatrix}
\begin{pmatrix}
d_1& 0 & 0 & \cdots \\
0 & d_2 & 0 & \cdots \\
0 &  0 & d_3 & \ddots \\
\vdots & \vdots & \ddots & \ddots
  \end{pmatrix}
\begin{pmatrix}
\bar t_{11} &\bar t_{21} &\bar t_{31}  &\dots  \\
\bar t_{12} & \bar t_{22} & \bar t_{32} & \dots \\
\bar t_{13} & \bar t_{23} & \bar t_{33} & \dots \\
\vdots & \vdots & \vdots &
  \end{pmatrix} \nonumber \\
&=\begin{pmatrix}
\sum d_j\vert t_{1j}\vert^2  &\sum d_j t_{1j}\bar t_{2j} &\sum d_j t_{1j}\bar t_{3j}  &\dots  \\
\sum d_j t_{2j}\bar t_{1j} & \sum d_j\vert t_{2j}\vert^2 & \sum d_j t_{2j}\bar t_{3j} & \dots \\
\sum d_j t_{3j}\bar t_{1j} & \sum d_j t_{3j}\bar t_{2j} & \sum d_j\vert t_{3j}\vert^2 & \dots \\
\vdots & \vdots & \vdots &
  \end{pmatrix} \nonumber
\end{align}
The condition that the off-diagonal coefficients of $TDT^*$ vanish for
every $D\in\Dc\cap\Fc(\Hc)$ is easily verified to be equivalent to
$$ \text{ if $j,k,\ell\ge 1$ and $k\ne\ell$, then $t_{kj}t_{\ell j}=0$. }$$
Now let $j\ge 1$ be fixed for the moment.
Since $T\in U(\Hc)$, we have $\Ker T=\{0\}$, hence there exists $k\ge 1$ for which $t_{kj}\ne0$,
and then the above condition implies $t_{\ell j}=0$ whenever $\ell\ne k$.
This shows that this subscript $k$ with $t_{kj}\ne0$ is uniquely determined by $j$, hence we may write $k=\sigma(j)$.
We thus obtain a function $\sigma\colon\{1,2,\dots\}\to\{1,2,\dots\}$
with the property that for every $j,k\ge 1$ we have $t_{kj}\ne0$ if and only if $k=\sigma(j)$.
Note that the function $\sigma$ is surjective, since otherwise there exists a row in $T$
consisting only of zeros, and then $\Ker T^*\ne\{0\}$.
Moreover $\sigma$ is injective since otherwise two columns in the matrix of $T$ would be linearly dependent,
which is not possible since $\Ker T=\{0\}$.
Consequently $\sigma$ is a bijection, that is, $\sigma\in\SS_\infty$. 
In summary, the unitary matrix for $T = (t_{ij})_{i,j\ge1}$  
has $t_{ij} = 0$ for all $i$ except $i=\sigma(j)$, 
in which case $\vert t_{\sigma(j) j}\vert = 1$.

If for $n\ge 1$ we choose $u_n=t_{\sigma(n)n}$, then $u_n\in\TT$ since $T$ is unitary
and $u_n$ is the unique nonzero entry on the $n$-th column of $T$. 
Hence $T=V_\sigma\diag\, u$. 

Finally, if  in addition $T\in U_{\Kc(\Hc)}$, then 
$\Kc(\Hc)\ni T-\1=V_\sigma\diag\,u-\1$, 
so $\vert u_{\sigma(n)}-u_n\vert=\Vert (T-\1)\bb_n\Vert\to 0$, 
so $\sigma(n)=n$ for all but finitely may~$n$, 
i.e.,  
$\sigma\in\SS_{\fin}$. 
\end{proof}

The Corollary~\ref{suppl4} below 
points to a rather
rich hierarchy of conjugacy classes
of Cartan subalgebras in operator ideals.
This is 
obtained as a consequence of the interaction between the lattice of operator ideals
and the notion of Cartan subalgebras and their conjugacy classes. 

\begin{remark}\label{suppl2}
\normalfont
For later use we record a couple of simple remarks on intersections of Cartan subalgebras
with smaller ideals.
\begin{enumerate}[(i)] 
\item\label{suppl2_item1}
For all operator ideals $\Ic\subseteq\Jc\subsetneqq\Bc(\Hc)$ and every subalgebra $\Cc\in\Cartan(\Jc)$
one has $\Cc\cap\Ic\in\Cartan(\Ic)$.
This follows directly by Proposition~\ref{descr} using the surjectivity of~$\Psi$.

\item\label{suppl2_item2}
If in addition $\Jc$ is a separable complete normed ideal, $\Cc\in\Cartan(\Jc)$, $\Ic\ne\{0\}$,
and another subalgebra $\Cc_0\in\Cartan(\Jc)$ has the property that 
the subalgebras $\Cc\cap\Ic$ and $\Cc_0\cap\Ic$
of $\Ic$ are $U_{\Ic}(\Hc)$-conjugated to each other, then $\Cc$ and $\Cc_0$ are $U_{\Ic}(\Hc)$-conjugated
to each other (hence $U_{\Jc}(\Hc)$-conjugated, since $U_{\Ic}(\Hc)\subseteq U_{\Jc}(\Hc)$).
Indeed, it easily follows that 
$\Cc\cap\Fc(\Hc) $ and $\Cc_0\cap\Fc(\Hc)$ 
are $\Jc$-norm-dense in $\Cc$ and $\Cc_0$, respectively. 
Since $\Fc(\Hc)\subseteq\Ic$, it then follows that also $\Cc\cap\Ic$ and $\Cc_0\cap\Ic$
are dense in $\Cc$ and $\Cc_0$, respectively.
Therefore, if $V\in U_{\Ic}(\Hc)$ for which $V(\Cc\cap\Ic)V^{-1}=\Cc_0\cap\Ic$,
via $\Vert\cdot\Vert_{\Jc}\ge\Vert\cdot\Vert$ and $\Jc$-norm limits one has also $V\Cc V^{-1}=\Cc_0$.
\end{enumerate}
\end{remark}

The next theorem negates the converse of Remark~\ref{suppl2}\eqref{suppl2_item2}.  
Its proof is based on a construction
that sharpens the one from the proof of Proposition~\ref{bases22},  
which would be recovered for the constant sequence $\theta=\langle 1,1,\dots\rangle$ 
in the proof of Theorem~\ref{suppl3} below. 
In fact, recalling Proposition~\ref{spectral}, 
one can see that Theorem~\ref{uncountable} is the limit situation for 
$\Jc=\Bc(\Hc)$ of Theorem~\ref{suppl3} below. 
Although there is a certain overlap of their proofs, 
we chose to present these theorems separately  
since the proof of Theorem~\ref{uncountable} is simpler 
and moreover for $\Jc=\Bc(\Hc)$ we did not define $\Cartan(\Jc)$ 
in order to avoid any confusion with the notion of Cartan subalgebras of von Neumann algebras  
as discussed in the Introduction. 

\begin{theorem}\label{suppl3}
For all operator ideals $\{0\}\subsetneqq\Ic
\subsetneqq\Jc\subsetneqq\Bc(\Hc)$ 
there exists an uncountable family of Cartan subalgebras of $\Jc$ that are 
pairwise $U_{\Jc}(\Hc)$-conjugated to each other but 
whose intersections with 
$\Ic$ are Cartan subalgebras of $\Ic$ that pairwise fail to be $U_{\Ic}(\Hc)$-conjugated to each other. 
\end{theorem}
It is easily seen in Theorem~\ref{suppl3} that the above family of Cartan subalgebras of $\Jc$ themselves pairwise fail to be
$\U_{\Ic}(\Hc)$-conjugated to each other, 
since so do their intersections with~$\Ic$.  

\begin{proof}[Proof of Theorem~\ref{suppl3}]
The proof proceeds in two steps. 
The first is a construction for a single operator.

Step 1. By using the surjectivity of $\Psi$ in Proposition~\ref{descr} let $\Cc:=\Dc\cap\Jc\in\Cartan(\Jc)$,
where we recall that this means that we fix an orthonormal basis $\bb=\{\bb_n\}_{n\ge 1}$ in $\Hc$
and we denoted by $\Dc$ the corresponding set of diagonal operators in $\Bc(\Hc)$ with respect to $\bb$ 
(see the beginning of Section~\ref{Prelims}). 
At this step we prove that if $\theta=\langle\theta_n\rangle_{n=1}^\infty$ is any decreasing sequence in $(0,\pi)$ 
with $\theta\in\Sigma(\Jc)\setminus\Sigma(\Ic)$ and $A_\theta\in U(\Hc)$ is the operator defined by 
\eqref{suppl3_proof_eq0.1}--\eqref{suppl3_proof_eq0.2} below, 
then $A\in U_{\Jc}(\Hc)\setminus U_{\Ic}(\Hc)$. 
Moreover, for the orthonormal basis $\vv=\{\vv_n\}_{n\ge 1}$ in $\Hc$ defined by
$v_n:=Ab_n$ for every $n\ge 1$,
if $\Dc(\vv)$ is the corrresponding set of diagonal operators in $\Bc(\Hc)$ with respect to $\vv$,
and we define $\Cc_0:=\Dc(\vv)\cap\Jc$, then 
$\Cc,\Cc_0\in\Cartan(\Jc)$ are $U_{\Jc}(\Hc)$-conjugated but fail to be $U_{\Ic}(\Hc)$-conjugated to each other.

First let $\theta=\langle\theta_n\rangle_{n=1}^\infty$ be any decreasing sequence in $(0,\pi)$
and denote $\lambda_n=e^{\ie\theta_n}\in\TT$ for every $n\ge 1$ and
$\lambda=\langle\lambda_n\rangle_{n=1}^\infty$.
For every integer $r\ge 1$ define the unitary operator
\begin{equation}\label{suppl3_proof_eq0.1}
A_{\theta_r}=\begin{pmatrix}
      \hfill \cos\theta_r & \sin\theta_r \\
      -\sin\theta_r & \cos\theta_r
      \end{pmatrix}
\text{ matricially and canonically on }\spann\{b_{2r-1},b_{2r}\}
\end{equation}
and then define the unitary block diagonal operator 
\begin{equation}\label{suppl3_proof_eq0.2}
A_\theta=\begin{pmatrix}
A_{\theta_1} & 0 & 0 & \cdots \\
 0  & A_{\theta_2} & 0 & \cdots \\
 0  & 0 & A_{\theta_3} & \ddots \\
 \vdots   & \vdots &  \ddots  & \ddots
    \end{pmatrix}\in U(\Hc).
\end{equation}
For arbitrary $r\ge 1$ the operator $A_{\theta_r}$ has the eigenvalues $\{\lambda_r,\bar\lambda_r\}$
and $\vert 1-\lambda_r\vert=2\sin\frac{\theta_r}{2}$.
Therefore, by using the fact that $\1-A_{\theta}$ is a normal operator,
it follows that for an arbitrary operator ideal $\Lc$ with characteristic set $\Sigma(\Lc)$ one has 
$$A_\theta\in U_{\Lc}\iff\Bigl\langle\sin\frac{\theta_r}{2}\Bigr\rangle_{r=1}^\infty\in\Sigma(\Lc) $$
and in particular
\begin{equation}\label{suppl3_proof_eq1}
A_\theta\in U_{\Jc}(\Hc)\setminus U_{\Ic}(\Hc)
\iff \Bigl\langle\sin\frac{\theta_r}{2}\Bigr\rangle_{r=1}^\infty\in\Sigma(\Jc)\setminus\Sigma(\Ic) 
\iff \theta\in\Sigma(\Jc)\setminus\Sigma(\Ic)
\end{equation}
where the last equivalence is based on $\lim\limits_{t\to0}\frac{\sin t}{t}=1$. 
Since 
$\Ic\subsetneqq\Jc$, we can choose the sequence $\theta\in\Sigma(\Jc)\setminus\Sigma(\Ic)$.  
This sequence will be fixed throughout this step 
of the proof 
and we will denote 
$A:=A_\theta\in U_{\Jc}(\Hc)\setminus U_{\Ic}(\Hc)$.

Now consider the orthonormal basis $v=\{v_n\}_{n\ge 1}$ in $\Hc$ defined by
$v_n=Ab_n$ for every $n\ge 1$,
let $\Dc(v)$ be the corrresponding set of diagonal operators in $\Bc(\Hc)$,
and define $\Cc_0:=\Dc(v)\cap\Jc$.

We now check that the assertions from this step of the proof 
hold true for $\Cc$ and $\Cc_0$.
First note that $\Cc,\Cc_0\in\Cartan(\Jc)$ by Proposition~\ref{descr}.
According to the way the orthonormal basis $v$ was constructed,
we have $A\Dc A^{-1}=\Dc(v)$. 
Since  $A\Jc A^{-1}= \Jc$, we obtain   
$A\Cc A^{-1}=A(\Dc\cap \Jc)A^{-1}
=(A\Dc A^{-1})\cap \Jc=\Dc(v) \cap \Jc=\Cc_0$.
This shows that the Cartan subalgebras
$\Cc$ and $\Cc_0$ of $\Jc$ are $U_{\Jc}(\Hc)$-conjugated to each other 
since $A\in U_{\Jc}(\Hc)$.

To prove $\Cc \cap \Ic$ and $\Cc_0 \cap \Ic$ are not $U_{\Ic}(\Hc)$-conjugated to each other, 
let us assume otherwise that the Cartan subalgebras
$\Cc\cap\Ic$ and $\Cc_0\cap\Ic$ of $\Ic$ were also $U_{\Ic}(\Hc)$-conjugated to each other. 
This means that there exists 
$B\in U_{\Ic}(\Hc)$  for which
$B(\Cc\cap\Ic)B^{-1}=\Cc_0\cap\Ic$.
On the other hand, since we have seen that $A\Cc A^{-1}=\Cc_0$,
one has also 
$A(\Cc\cap\Ic)A^{-1}=\Cc_0\cap\Ic$.
It then follows that 
$B(\Cc\cap\Ic)B^{-1}=A(\Cc\cap\Ic)A^{-1}$,
hence
$$(A^{-1}B)(\Cc\cap\Ic)(A^{-1}B)^{-1}=\Cc\cap\Ic.$$
By using Lemma~\ref{suppl1} we now obtain a permutation $\sigma\in\SS_\fin$ and
a sequence $u=\langle u_n\rangle_{n=1}^\infty$ with $u_n\in\TT$ for every $n\ge 1$,
for which $A^{-1}B=V_\sigma(\diag\, u)$,
hence
\begin{equation}\label{suppl3_proof_eq2}
B(\diag\, u)^{-1}=AV_\sigma.
\end{equation}
Recall that 
$B\in U_{\Ic}(\Hc)$,
hence 
$B-\1\in\Ic$, and this implies implies 
\begin{equation}\label{suppl3_proof_eq3}
A-(\diag\,u)^{-1}V_\sigma^{-1}\in\Ic
\end{equation}
since $A-(\diag\,u)^{-1}V_\sigma^{-1}=(B-\1)(\diag\,u)^{-1}V_\sigma^{-1}$ by~\eqref{suppl3_proof_eq2}. 
 
On the other hand, since $\sigma\in\SS_{\fin}$, it follows that  
$V_\sigma =$ identity on all $b_n$ except for finitely many,
and so the subtraction of $(\diag\,u)^{-1}V_\sigma^{-1}$ acts only on the diagonal of $A$, 
except for finitely many blocks of $A$. 
In other words, this difference
$A-(\diag\,u)^{-1}V_\sigma^{-1}$
is a matrix which is block diagonal with off diagonal entries $\sin\frac{\theta}{2}$ in all but finitely many of the blocks.
Now each $2\times 2$ block then has norm $\ge\vert \sin\frac{\theta}{2}\vert$, 
and so also its biggest $s$-number.
Using basic $s$-number theory for operator ideals it follows that the $s$-numbers of $A-(\diag\,u)^{-1}V_\sigma^{-1}$
 are bounded below by a constant multiple of the sequence $\theta$ 
 (see \cite[Subsect.\ 2.7]{DFWW04}). 
Hence the positive operator $\vert A-(\diag\,u)^{-1}V_\sigma^{-1}\vert$ 
is bounded below by the direct sum of $2\times 2$ positive matrices with eigenvalues $\vert\sin\frac{\theta}{2}\vert$ and~$0$.
Since $\theta\in\Sigma(\Jc)\setminus\Sigma(\Ic)$, 
it then follows that $A-(\diag\,u)^{-1}V_\sigma^{-1}\in\Jc\setminus\Ic$. 
This contradicts \eqref{suppl3_proof_eq3} and hence completes the proof that
$\Cc\cap\Ic$ and $\Cc_0\cap\Ic$
fail to be $U_{\Ic}(\Hc)$-conjugated to each other.

Step 2. Let $\theta=\langle\theta_n\rangle_{n=1}^\infty$ is any decreasing sequence in $(0,\pi)$ 
with $\theta\in\Sigma(\Jc)\setminus\Sigma(\Ic)$. 
Then for arbitrary  
$\alpha\in(0,1]$ the sequence 
$\alpha\theta:=\langle\alpha\theta_n\rangle_{n=1}^\infty$ 
is again a decreasing sequence in $(0,\pi)$ 
with $\alpha\theta\in\Sigma(\Jc)\setminus\Sigma(\Ic)$, 
hence $A_{\alpha\theta}\in U_{\Jc}(\Hc)\setminus U_{\Ic}(\Hc)$ 
by Step~1 of the present proof. 
Moreover, since $\alpha\theta+\beta\theta=(\alpha+\beta)\theta$, 
we obtain $A_{\alpha\theta}A_{\beta\theta}=A_{(\alpha+\beta)\theta} $ if $\alpha,\beta\in(0,1]$ with $\alpha+\beta\le 1$, 
by using \eqref{suppl3_proof_eq0.1}--\eqref{suppl3_proof_eq0.2}.    
This shows that for the orthonormal bases $v^{(\alpha)}=\{v_n^{(\alpha)}\}_{n\ge 1}$ in $\Hc$ defined by
$\vv_n^{(\alpha)}:=A_{\alpha\theta}b_n$ for every $n\ge 1$, 
we have $A_{\alpha\theta}\vv^{(\beta)}=\vv^{(\alpha+\beta)}$ (vectorwise on these bases) if $\alpha+\beta\le m$. 

Thus any two distinct bases in the set $\{\vv^{(\alpha)}\mid \alpha\in(0,1]\}$ 
are related to each other by a suitable unitary operator of the form $A_{t\theta}\in U_{\Jc}(\Hc)\setminus U_{\Ic}(\Hc)$ 
with $t\in(0,1]$, 
just as it was the case with the bases $\bb$ and $\vv$ in Step 1 of this proof. 
Hence, if for every basis $v^{(\alpha)}$ one denotes by $\Dc(v^{(\alpha)})$ 
its corresponding set of diagonal operators in $\Bc(\Hc)$,
and we define $\Cc_\alpha:=\Dc(v^{(\alpha)})\cap\Jc$, then 
$\{\Cc_\alpha\mid \alpha\in(0,1]\}\subseteq\Cartan(\Jc)$ are pairwise $U_{\Jc}(\Hc)$-conjugated and 
their intersections with $\Ic$ pairwise fail to be $U_{\Jc}(\Hc)$-conjugated to each other.
\end{proof}

\begin{corollary}\label{suppl4}
For all operator ideals 
$\{0\}\subsetneqq\Ic
\subsetneqq
\Jc\subsetneqq\Bc(\Hc)$
and every $U_{\Jc}(\Hc)$-conjugacy class $\Oc\subseteq\Cartan(\Jc)$, 
the set $\{\Cc\cap\Ic\mid\Cc\in\Oc\}\subseteq\Cartan(\Ic)$
disjointly partitions into an uncountable family of $U_{\Jc}(\Hc)$-conjugacy class
of Cartan subalgebras of~$\Ic$.
\end{corollary}

\begin{proof}
The assertion follows directly from Theorem~\ref{suppl3}.
\end{proof}

\begin{remark}\label{suppl5}
\normalfont 
Using Remark~\ref{equivar}, one can 
translate information provided by Theorems \ref{uncountable} and \ref{suppl3} 
from the conjugacy action of unitary groups on Cartan subalgebras to  
the natural action $\alpha$ of the same groups on $\Decomp(\Hc)$ (see Definition~\ref{action}).  
One thus obtains that for arbitrary operator ideals $\{0\}\subsetneqq\Ic
\subsetneqq\Jc\subseteq\Bc(\Hc)$ one has $U_{\Ic}(H)\subsetneqq U_{\Jc}(H)$, 
and for any $\Sc_0\in\Decomp(\Hc)$, the $U_{\Jc}(H)$-orbit $\alpha(U_{\Jc}(H)\times\{\Sc_0\})$  
partitions into uncountably many $U_{\Ic}(H)$-orbits. 
 
To see this, 
first use Proposition~\ref{descr} for the ideal $\Jc$ to construct 
$\Cc_0:=\Psi(\Sc_0)\in\Cartan(\Jc)$ with its $U_{\Jc}(\Hc)$-conjugacy class $\Oc$. 
It follows by Theorems \ref{uncountable} and \ref{suppl3} that there exists 
an uncountable family $\{\Cc_i\}_{i\in I}$ of elements of $\Oc$ 
that pairwise fail to be $U_{\Ic}(\Hc)$-conjugate to each other. 
Then $\{\Psi^{-1}(\Cc_i)\}_{i\in I}$ is an uncountable family of elements of $\Decomp(\Hc)$ 
that (because of the commutative diagram from Remark~\ref{equivar}) 
belong to the $U_{\Jc}(\Hc)$-orbit of $\Sc_0\in\Decomp(\Hc)$ and yet 
pairwise fail to be mapped to each other by the action of $U_{\Ic}(H)$ on  $\Decomp(\Hc)$ via~$\alpha$. 
Thus there exist uncountably many distinct $U_{\Ic}(H)$-orbits contained in 
the $U_{\Jc}(H)$-orbit $\alpha(U_{\Jc}(H)\times\{\Sc_0\})$. 
\end{remark}

\section{Differentiable structures on conjugacy classes of Cartan subalgebras}

In order to state the next proposition, we recall that for every unital $C^*$-algebra $\Ac$,  
its unitary group $U_{\Ac}:=\{u\in\Ac\mid u^*u=uu^*=\1\}$ has the structure of a Banach-Lie group.  
Moreover, for every Banach-Lie group $G$ and every automorphism group $S$ of $G$,  
one can define the semidirect product $S\ltimes G$, 
and this has a unique Banach-Lie group structure for which $G$ is an open subgroup. 
See \cite{Be06} for background information on these constructions of Banach-Lie groups. 

We now frame this in the specific setting to be used in this section. 
There is a natural action
of the permutation group $\SS_\infty$ by $*$-automorphisms of the $C^*$-algebra $\ell^\infty(\NN)$, 
that is, every $\sigma\in\SS_\infty$ defines the $*$-automorphism $\alpha_\sigma\colon \ell^\infty(\NN)\to\ell^\infty(\NN)$ 
by $\alpha_\sigma(z)=\langle z_{\sigma^{-1}(n)}\rangle_{n=1}^\infty$ for every 
$z=\langle z_n\rangle_{n=1}^\infty\in\ell^\infty(\NN)$. 
One has $\alpha_{\sigma\tau}=\alpha_{\sigma}\alpha_{\tau}$ for all $\sigma,\tau\in\SS_\infty$. 

By restricting this action of $\SS_\infty$ from the $C^*$-algebra $\ell^\infty(\NN)$ 
to its unitary group $U(\ell^\infty(\NN))\simeq\TT^{\NN}$, 
where $\TT:=\{z\in\CC\mid \vert z\vert=1\}$, 
we obtain an action of $\SS_\infty$ by automorphisms of the group $U(\ell^\infty(\NN))$. 
From this we can then construct the semidirect product $\SS_\infty\ltimes U(\ell^\infty(\NN))$, 
whose underlying set is the Cartesian product $\SS_\infty\times U(\ell^\infty(\NN))$ 
and whose group operation is  
$$(\sigma,z)\cdot(\tau,w)=(\sigma\tau,\alpha_{\tau^{-1}}(z)w) $$
for all $\sigma,\tau\in \SS_\infty$ and $z,w\in U(\ell^\infty(\NN))$. 
Since $\ell^\infty(\NN)$ is a unital $C^*$-algebra, 
its unitary group $U(\ell^\infty(\NN))$ is a Banach-Lie group. 
Moreover, by using the discrete topology of $\SS_\infty$, we endow the semidirect product $\SS_\infty\ltimes U(\ell^\infty(\NN))$ with the structure of a Banach-Lie group for which the connected component of the identity element is $U(\ell^\infty(\NN))$. 

\begin{proposition}\label{isotropy}
If $\Sc\in\Decomp(\Hc)$, then 
$\U(\Hc)_{\Sc}$
and the semidirect product $\SS_\infty\ltimes U(\ell^\infty(\NN))$ 
are isomorphic Banach-Lie groups.  
\end{proposition}

\begin{proof}
Recall from Definition~\ref{action} that $U(\Hc)_{\Sc}$ is the class of unitaries that fix 
the orthogonal one dimensional decomposition $\Sc$ of $\Hc$.
Let us consider any labeling of the elements of $\Sc$ by natural numbers,
$\Sc=\{\Vc_n\mid n\in\NN\}$, and for every $n\in\NN$ pick a unit vector $\vv_n\in\Vc_n$.
If $W\in\U(\Hc)_{\Sc}$, then we obtain a bijection of $\Sc$ given by $\Vc\mapsto W(\Vc)$,
which in turn defines a permutation $\sigma(W)\in\SS_\infty$.
It follows that the unitary operator $\widetilde{W}:=V_{\sigma(W)}W$ has the property
$\widetilde{W}(\Vc)=\Vc$ for every $\Vc\in\Sc$,
hence for every $n\in\NN$ there exists $z_n(W)\in\CC$ such that $\vert z_n(W)\vert=1$
and $\widetilde{W}\vv_n=z_n(W)\vv_n$.
Then $z(W)=\{z_n(W)\}_{n\in\NN}$ is a unitary element in the $C^*$-algebra $\ell^\infty(\NN)$,
and it is easily checked that the mapping $W\mapsto (\sigma(W),z(W))$ is a group isomorphism 
as in the statement of the proposition.
\end{proof}

In the following statement, for every Cartan subalgebra $\Cc$ of a proper operator ideal $\Ic$,
we denote by $\Dc(\Cc)$ the set of all operators in $\Bc(\Hc)$ which are diagonal with respect to $\Psi^{-1}(\Cc)\in\Decomp(\Hc)$. 
Equivalently, $\Dc(\Cc)$ is the set of all operators in $\Bc(\Hc)$ that commute with every operator in $\Cc$, 
or also, $\Dc(\Cc)$ is the closure of $\Cc$ with respect to the weak operator topology. 

\begin{corollary}\label{isotropy_id}
Let $\Ic$ be a proper operator ideal in $\Bc(\Hc)$ and $\Cc\in\Cartan(\Ic)$, and consider
the corresponding normalizer subgroup of the group $U_{\Ic}(\Hc)$, 
$$U_{\Ic,\Cc}(\Hc):=\{W\in U_{\Ic}(\Hc)\mid W\Cc W^{-1}=\Cc\}.$$
Then the groups 
$$U_{\Ic,\Cc}(\Hc)\text{ and } \SS_{\fin}\ltimes (U(\Hc)\cap(\1+(\Dc(\Cc)\cap\Ic)))$$
are isomorphic, where the latter semidirect product is defined by using the natural action of
the group of finite permutations $\SS_{\fin}$ on the group of unitary operators in $\1+\Ic$ which are diagonal with respect to
$\Psi^{-1}(\Cc)\in\Decomp(\Hc)$. 
If moreover $\Ic$ is a complete normed ideal, then the above groups are isomorphic Banach-Lie groups. 
\end{corollary}

\begin{proof}
Use Proposition~\ref{isotropy} along with Proposition~\ref{S_fin}.
\end{proof}

\begin{theorem}\label{smooth}
Assume the proper ideal $\Ic$ in $\Bc(\Hc)$ is symmetrically normed and
its underlying Banach space is the dual of a Banach space.
Then for every $\Cc\in\Cartan(\Ic)$ the corresponding isotropy group $U_{\Ic,\Cc}(\Hc)$
is a Banach-Lie subgroup of $U_{\Ic}(\Hc)$ and the $U_{\Ic}(\Hc)$-orbit with respect to the group action $\beta$
has a structure of a smooth homogeneous space of the Banach-Lie group $U_{\Ic}(\Hc)$.
Moreover, the Banach manifolds obtained in this way for various choices of $\Cc\in\Cartan(\Ic)$
are diffeomorphic to each other.
\end{theorem}

\begin{proof}
Let us consider the complex associative Banach $*$-algebra $\Bg=\CC\1+\Ic$ endowed with the
norm given by $\Vert z\1+T\Vert=\vert z\vert+\Vert T\Vert_{\Ic}$,
and denote by $\Bg^\times$ its group of invertible elements.
Define the bounded linear functional $\psi\colon\Bg\to\CC$, $\psi(z\1+T)=z$.
If we denote by $\{P_n\}_{n\ge 1}$ the family of orthogonal projections
corresponding to the 1-dimensional subspaces in $\Psi^{-1}(\Cc)\in\Decomp(\Hc)$, 
then it is easily seen that 
$$U_{\Ic,\Cc}(\Hc)=\{V\in\Bg^\times\mid \psi(V)=1,\, V^*V=\1,\, (\forall n\ge 1)\ [VP_nV^{-1},P_n]=0\},$$
the latter condition being equivalent to the $V$-change of basis preserving the 1-dimensional subspaces.
This observation shows that $U_{\Ic,\Cc}(\Hc)$ is an algebraic subgroup of $\Bg^\times$ of degree $\le 2$ 
(see Definition~\ref{algebraic} below)
and therefore it follows 
that $U_{\Ic,\Cc}(\Hc)$
has the structure of a Banach-Lie group whose underlying topology is
its topology inherited from $\Bg$  
(see for instance \cite[Th. 4.13]{Be06} and its proof).
On the other hand, the Lie algebra of $U_{\Ic,\Cc}(\Hc)$
is equal to the Lie algebra of any of its open subgroups.
Since $\SS_{\fin}$ is a discrete group, it follows by Corollary~\ref{isotropy_id}
that $U(\Hc)\cap(\1+(\Dc(\Cc)\cap\Ic))$ is an open subgroup of $U_{\Ic,\Cc}(\Hc)$,
hence we obtain at last the following description for the Lie algebra of $U_{\Ic,\Cc}(\Hc)$:
\begin{equation}\label{smooth_proof_eq1}
\ug_{\Ic,\Cc}(\Hc)=\{T\in\Ic\mid T^*=-T,\, (\forall n\in\NN)\ TP_n=P_nT\}.
\end{equation}
Recall from \cite[Prop. 9.28]{Be06} that the Lie algebra of $U_{\Ic}(\Hc)$ is
$$\ug_{\Ic}(\Hc)=\{T\in\Ic\mid T^*=-T\}.$$
In order to prove that $U_{\Ic,\Cc}(\Hc)$ is a Banach-Lie subgroup of $U_{\Ic}(\Hc)$,
we still need to check that the closed real linear subspace $\ug_{\Ic,\Cc}(\Hc)$
has a direct complement in $\ug_{\Ic}(\Hc)$ 
(see the \cite[Def. 4.1(iii)]{Be06}).

In view of \cite[Th. 3.1(a)]{BP07},
it suffices to construct a commutative semigroup of contractive linear maps on $\ug_{\Ic}(\Hc)$
whose fixed point set is precisely $\ug_{\Ic}(\Hc)$.
To this end, note that the condition $TP_n=P_nT$ from \eqref{smooth_proof_eq1}
is equivalent to the family of fixed point equations $e^{isP_n}Te^{-isP_n}=T$ for all $s\in\RR$.
Now for every $n\ge 1$ and $s\in\RR$ define
$$\gamma_{s,n}\colon\Ic\to\Ic,\quad \gamma_{s,n}(T)=e^{isP_n}Te^{-isP_n}$$
and set
$$\Gamma=\{\gamma_{s_1,n_1}\circ\dots\circ\gamma_{s_k,n_k}\mid k\ge 1,\, n_1,\dots,n_k\ge 1,\, s_1,\dots,s_k\in\RR\},$$
which is an abelian group.
The above observations show that
$$\ug_{\Ic,\Cc}(\Hc)=\{T\in\ug_{\Ic}(\Hc)\mid (\forall \gamma\in\Gamma)\quad \gamma(T)=T\}.$$
Since $\Ic$ is a symmetrically normed ideal, it follows that for every $\gamma\in\Gamma$ we have
$\Vert\gamma\Vert=1$. 
Hence we can use \cite[Th. 3.1(a)]{BP07} to obtain a bounded linear operator 
$P\colon \ug_{\Ic}(\Hc)\to \ug_{\Ic}(\Hc)$ with $P^2=P$ and $\Ran P=\ug_{\Ic,\Cc}(\Hc)$.  
Then $\ug_{\Ic,\Cc}(\Hc)+\Ker P=\ug_{\Ic}(\Hc)$ and $\ug_{\Ic,\Cc}(\Hc)\cap\Ker P=\{0\}$, 
that is, $\Ker P$ 
is a direct complement for $\ug_{\Ic,\Cc}(\Hc)$.

This completes the proof of the fact that $U_{\Ic,\Cc}(\Hc)$ is a Banach-Lie subgroup of $U_{\Ic}(\Hc)$,
and then the quotient $U_{\Ic}(\Hc)/U_{\Ic,\Cc}(\Hc)$ has the natural structure of a smooth homogeneous space 
(see for instance \cite[Th. 4.19]{Be06}).
On the other hand, recall that $U_{\Ic,\Cc}(\Hc)$ is the isotropy group for the action of $U_{\Ic}(\Hc)$
on the $U_{\Ic}(\Hc)$-conjugacy class of the Cartan subalgebra $\Cc$ of $\Ic$. 
Hence we have a canonical $U_{\Ic}(\Hc)$-equivariant bijection from that conjugacy class onto $U_{\Ic}(\Hc)/U_{\Ic,\Cc}(\Hc)$. 
This makes the conjugacy class into a smooth homogeneous space.

Finally, if $\Cc_1,\Cc_2\in\Cartan(\Ic)$, then 
by Remark~\ref{spectral} 
there exists a unitary operator $V\in U(\Hc)$ such that $\beta_V(\Cc_1)=V\Cc_1V^{-1}=\Cc_2$.
The mapping $X\mapsto VXV^{-1}$ also gives an automorphism of the Banach-Lie group $U_{\Ic}(\Hc)$
that maps the isotropy group $U_{\Ic,\Cc_1}(\Hc)$ onto the isotropy group $U_{\Ic,\Cc_2}(\Hc)$,
and therefore gives rise to a diffeomorphism $U_{\Ic}(\Hc)/U_{\Ic,\Cc_1}(\Hc)\simeq U_{\Ic}(\Hc)/U_{\Ic,\Cc_2}(\Hc)$.
Therefore the $U_{\Ic}(\Hc)$-conjugacy classes of $\Cc_1$ and $\Cc_2$ are diffeomorphic to each other, and this concludes the proof.
\end{proof}

\begin{remark}
\normalfont
On the applicability of Theorem~\ref{smooth} we note the following.
It follows by the classical duality theory of symmetrically normed ideals
that there exist a lot of such ideals which are duals of Banach spaces,
including for instance the Schatten ideals $\Sg_p(\Hc)$ with $1\le p<\infty$;
see \cite[Sect. III.12]{GK69}, and also \cite[Sect 9.4]{Be06}.  
However, Theorem~\ref{smooth} cannot be applied for $\Ic=\Kc(\Hc)=\Sg_\infty(\Hc)$ 
since it is well known that this ideal has no predual.
\qed
\end{remark}

\begin{remark}
\normalfont
We now summarize some results on conjugacy classes of maximal abelian self-adjoint subalgebras
of three types of algebras:
matrix algebras, operator ideals, and $C^*$-algebras.
This summary may be helpful in placing
Theorem~\ref{smooth} in some perspective.
\begin{itemize}
\item As already mentioned in the Introduction, if we replace in Definition~\ref{action}
the ideal $\Ic$ by a matrix algebra $M_n(\CC)$, then the action $\beta$ is transitive,
that is, it has only one orbit~(see Theorem~\ref{first_th} for a more general result).
\item It follows by Theorem~\ref{uncountable} that if $\Ic$ is a proper operator ideal in $\Bc(\Hc)$,
then the action $\beta$ restricted to the group $U_{\Ic}(\Hc)$ has uncountably many orbits.
\item If we replace in Definition~\ref{action}
the ideal $\Ic$ by a $C^*$-algebra of operators, then we still find infinitely many orbits,
however there may be only countably many ones; 
see 
\cite[Rem. 4.2]{BPW13b} 
for $\Ic=\Bc(\Hc)$.
\end{itemize}
\qed
\end{remark}

\begin{remark}
\normalfont
The method of proof of Theorem~\ref{smooth} can be used for constructing smooth structures on
the $U_{\Ic}(\Hc)$-orbits of other actions as well.
For instance, this is the case with the action
$$U_{\Ic}(\Hc)\times\Bc(\Hc)^{\sa}\to\Bc(\Hc)^{\sa},\quad (V,X)\mapsto VXV^{-1}$$
(see \cite[Th. 4.33]{Be06} for a related result).
\qed
\end{remark}

\section{Further problems}

\begin{problem}
\normalfont
Find a parameterization of the set of all $U_{\Ic}(\Hc)$-conjugacy classes of Cartan subalgebras of~$\Ic$, 
simpler than the bijective map~\eqref{equivar_eq1} in Remark~\ref{equivar}.
For instance, if $\Ic=\Bc(\Hc)$, then  
\cite[Rem. 4.2]{BPW13b} 
shows that $\{0,1,2,\dots\}\cup\{-\infty,\infty\}$ is such a set of parameters.
\end{problem}

\begin{problem}
\normalfont
A question whose affirmative answer would provide extra motivation for calling the masa's of
an operator ideal $\Ic$ as Cartan subalgebras of $\Ic$: 
Does there exist any generalization of Theorem~\ref{suppl6} beyond the class of complete separable norm ideals? 

More specifically, let $\Cc\in\Cartan(\Ic)$ be the set of all diagonal operators in $\Ic$
with respect to some fixed orthonormal basis in $\Hc$.
Then let $\Ac$ be the set of all operators given by matrices with the entry 1 on some off-diagonal position and 0 elsewhere.
Can we make sense of the assertion that every operator in $\Ic$ is an infinite linear combination of elements in $\Ac\cup\Dc$?

The problem is the convergence of such a linear combination. 
However, the main point would be to make sense of the assertion of Theorem~\ref{suppl6} for arbitrary ideals,
and in this case one should perhaps expect to have ``infinite linear combinations'' somehow similar to the formal power series,
which are also convergent in a suitable sense expressed in purely algebraic terms.
Such a suitable sense could be perhaps imagined for arbitrary operator ideals, by using the characteristic sets.
Note that we do not mean to sum up anything inside of the diagonal algebra $\Dc$, but rather to make sense of the fact that the farther an entry is from the diagonal, the smaller it should be thought of.

So the point is that main information of a matrix is concentrated on the main diagonal, and when we get far from the main diagonal, the information carried by the off-diagonal entries is less and less significant.
This would provide an infinite-dimensional version of the root-space decomposition that occurs in the proof of  \cite[Prop. 4.16]{BPW13a}.
\end{problem}
 
\begin{problem}
\normalfont
On Remark~\ref{equiv_bases},   
where for $e=\{e_n\}_{n\ge 1}$ and $f=\{f_n\}_{n\ge 1}$ orthonormal bases in $\Hc$,
we say they are $\Ic$-equivalent to each other if there exists $W\in U_{\Ic}(\Hc)$
such that $We_n=f_n$ for all $n\ge 1$.
It would be useful to study necessary or sufficient conditions of $\Ic$-equivalence, as for instance 
the following ones for the largest/smallest proper operator ideal: 
\begin{itemize}
\item In the compact case $\Ic = \Kc(\Hc)$, $e$ and $f$ are $\Ic$-equivalent to each other 
if and only if $\Vert e_n-f_n\Vert\to0$.  
\item And for the finite ranks $\Ic = \Fc(\Hc)$, the $\Ic$-equivalence condition is 
$e_n = f_n$ for all but finitely many $n$. 
\end{itemize}
In fact, the above condition in the case $\Ic = \Fc(\Hc)$ is clearly sufficient, 
however it is not difficult to see that it is not necessary. 
To this end define $v:=\sum\limits_{n\ge1}\alpha_n e_n\in\Hc$ with $\Vert v\Vert=1$, 
where $\langle\alpha_n\rangle_{n=1}^\infty\in\ell^2$ is any sequence with $\alpha_n\ne0$ 
for all but finitely many subscripts, for instance $\alpha_n=\frac{1}{2^{n/2}}$ for every $n\ge 1$. 
If $P:=(\cdot,v)v\in\Bc(\Hc)$ is the orthogonal projection onto the 1-dimensional subspace spanned by $v$,  
then $V:=\1-2P$ is a unitary (self-adjoint) operator, 
hence $\{f_n:=Ve_n\mid n\ge 1\}$ is an othonormal basis in $\Hc$. 
Since $V\in\1+\Fc(\Hc)$, one has $V\in U_{\Fc(\Hc)}$, hence 
the orthonormal bases $e=\{e_n\}_{n\ge 1}$ and $f=\{f_n\}_{n\ge 1}$ are 
$\Fc(\Hc)$-equivalent. 
However, $e_n-f_n=2Pe_n=2\bar\alpha_n v\ne 0$ for infinitely many $n$. 

Perhaps using a slightly more complicate example, one could also prove that the above 
assertion on $\Kc(\Hc)$-equivalent bases also needs extra assumptions in order to hold true.   
\end{problem}

\begin{problem}
\normalfont
Can one obtain sufficient conditions for $U_{\Ic}(\Hc)$-diagonalizability 
of normal operators in a $\Bc(\Hc)$-ideal $\Ic$? 
See Remark~\ref{suffcond} for the case when $\Ic$ is the Hilbert-Schmidt ideal, 
and also \cite[Prop. 4.3]{BPW13b} for the limit cases when $\Ic$ is the ideal of finite-rank operators 
(when no extra condition is needed) or when $\Ic=\Bc(\Hc)$. 
The case when $\Ic$ is the trace class would be particularly interesting. 
\end{problem}

\begin{problem}
\normalfont
What information does Proposition~\ref{July6_cor}\eqref{July6_cor_item3} provide for the size of $\Dc_{\Ic}$, 
beyond the fact that $\Pi^{-1} \Dc_{\Ic} \subset \Ic^2$?
\end{problem}
 
\begin{problem}
\normalfont
On the structure of \emph{some subsets} of $\Dc_{\Ic}$ and $\Dc_\Ic^{\sa}$,  
namely for operators with spectral multiplicities one, 
we know from Proposition~\ref{7april2014} that they are both closed under finite permutations, 
and in fact the equivalence relation of one diagonal being a finite permutation of the other partitions 
both $\Dc_{\Ic}$ and $\Dc_\Ic^{\sa}$.
It would be perhaps interesting to see in what form 
the above mentioned proposition generalizes to arbitrary normal compact operators 
without any multiplicity restriction. 
\end{problem}

\appendix

\section{Lie theory for some infinite-dimensional algebraic groups}\label{sect3}

\subsection*{Infinite-dimensional linear algebraic reductive groups}
The notion of linear algebraic group in infinite dimensions requires
the following terminology.
If $\Ac$ is a real Banach space,
then a vector-valued continuous polynomial function on $\Ac$ of degree $\le n$
is a function $p\colon\Ac\to\Vc$, where $\Vc$ is another real Banach space,
such that for some continuous multilinear maps
$$\psi_k\colon
\underbrace{\Ac\times\cdots\times\Ac}_{\text{$k$ times}}
\to\Vc$$
(for $k=0,1,\ldots,n$) we have
$p(a)=\psi_n(a,\ldots,a)+\cdots+\psi_1(a)+\psi_0$
for every $a\in\Ac$, where $\psi_0\in\Vc$.

Now let $\Bg$ be a real \textit{associative} unital Banach algebra,
hence a real Banach space endowed with a bounded bilinear mapping $\Bg\times\Bg\to\Bg$, $(x,y)\mapsto xy$ 
which is associative and admits a unit element $\1\in\Bg$.
Then the set
$$\Bg^\times:=\{x\in\Bg\mid(\exists y\in\Bg)\ xy=yx=\1\}$$
is an open subset of $\Bg$ and
has the natural structure of a Banach-Lie group \cite[Example 6.9]{Up85}.
The Lie algebra of $\Bg^\times$ is
again the Banach space $\Bg$, viewed however as a \textit{nonassociative} Banach algebra,
more precisely as a Banach-Lie algebra whose Lie bracket is the bounded bilinear mapping
$\Bg\times\Bg\to\Bg$, $(x,y)\mapsto xy-yx$.

\begin{definition}\label{closed}
\normalfont
If $\Bg$ is a real associative unital Banach algebra and
$G$ is a closed subgroup of $\Bg^\times$, then
the \textit{Lie algebra of $G$} is
$$\gg:=\{x\in\Bg\mid(\forall t\in\RR)\quad\exp(tx)\in G\}.$$
\end{definition}

\begin{remark}\label{closed_alg}
\normalfont
In the setting of Definition~\ref{closed},
the set $\gg$ is a closed Lie subalgebra of $\Bg$
\cite[Corollary 6.8]{Up85}.

In fact, since $G$ is a closed subset of $\Bg^\times$, it is easily seen that $\gg$ is closed in $\Bg$.
Moreover, by using the well-known formulas \cite[Proposition 6.7]{Up85}
$$\begin{aligned}
\exp(t(x+y))&=\lim_{k\to\infty}\Bigl(\exp(\frac{t}{k}x)\exp(\frac{t}{k}y)\Bigr)^k, \\
\exp(t^2[x,y])&=\lim_{k\to\infty}\Bigl(\exp(\frac{t}{k}x)\exp(\frac{t}{k}y)\exp(-\frac{t}{k}x)\exp(-\frac{t}{k}y)\Bigr)^{k^2}
\end{aligned}$$
which hold true for all $x,y\in\Bg$ and $t\in\RR$,
it follows that for every $x,y\in\gg$ we have $x+y\in\gg$ and $[x,y]\in\gg$.
Then it is easy to check that $\gg$ is a real linear subspace of $\Bg$.

Moreover, if $\Bg$ is endowed with a continuous involution such that for every $b\in G$ we have $b^*\in G$,
then for every $x\in\gg$ we have $x^*\in\gg$.
\end{remark}

\begin{definition}[\cite{HK77}]\label{algebraic}
\normalfont
Let $\Bg$ be a real  
associative unital Banach algebra, $n$ a positive integer,
and $G$ a subgroup of $\Bg^\times$.
We say that $G$ is an
\textit{algebraic group in $\Bg$ of degree $\le n$}
if we have
$$G=\{b\in \Bg^\times\mid (\forall p\in\Pc)\quad p(b,b^{-1})=0\}$$
for some set $\Pc$ of vector-valued continuous polynomial functions of degree $\le n$ on $\Bg\times\Bg$ 
(see \cite[Sect. 3]{BPW13b} for more details and examples).
Note that $G$ is a closed subgroup of $\Bg^\times$, hence its Lie algebra can be defined
as in Definition~\ref{closed}.

If moreover $\Bg$ is endowed with a continuous involution $b\mapsto b^*$
and for every $b\in G$ we have $b^*\in G$, then we say that the group $G$ is \textit{reductive}.
\end{definition}

\begin{definition}\label{algebraic_restr}
\normalfont
Let $\Bg$ be a real 
associative unital Banach algebra, $n$ a positive integer,
and $G$ an algebraic subgroup of $\Bg^\times$ of degree $\le n$
with the Lie algebra $\gg$ ($\subseteq\Bg$).
Then for every one-sided ideal $\Ig$ of $\Bg$,
the corresponding
\textit{$\Ig$-restricted} algebraic group is 
$$G_{\Ig}:=G\cap(\1+\Ig)$$
and the \textit{Lie algebra of $G$} is $\gg_{\Ig}:=\gg\cap\Ig$. 
\end{definition}

\begin{remark}\label{simple}
\normalfont
Here are some simple remarks
on algebraic structures that occur in the preceding definition. 
Let $\Bg$ be a unital ring and $\Ig$ be a one-sided ideal of $\Bg$.
\begin{enumerate}[(i)] 
\item\label{simple_item1} The set
$\Bg^\times\cap(\1+\Ig)$ is always a subgroup of the group $\Bg^\times$
of invertible elements in $\Bg$.

To see this, let us assume for instance that $\Ig\Bg\subseteq\Ig$.
Then $\Ig\Ig\subseteq\Ig$, hence $(\1+\Ig)(\1+\Ig)\subseteq\1+\Ig$,
and thus $(\1+\Ig)\cap\Bg^\times$ is closed under the product.
On the other hand, if $x\in\Ig$, $b\in\Bg$ and $(\1+x)b=\1$, then $x=\1-xb\in\1+\Ig\Bg\subseteq\1+\Ig$,
hence $(\1+\Ig)\cap\Bg^\times$ is also closed under the inversion.

\item\label{simple_item2}
By definition, every one-sided ideal of a real algebra is assumed to be a real linear subspace.
Therefore, if the unital ring $\Bg$ has the structure of a real algebra,
then $\Ig$ is an associative subalgebra of $\Bg$ and in particular $\Ig$
has the natural structure of a real Lie algebra with the Lie bracket
defined by $[x,y]:=xy-yx$ for all $x,y\in\Ig$.

\item\label{simple_item3} If $\Bg$ is a ring endowed with an involution $b\mapsto b^*$ and $\Ig$ is
a self-adjoint one-sided ideal of $\Bg$, then $\Ig$ is actually a two-sided ideal.

In fact, if we assume for instance $\Ig\Bg\subseteq\Ig$,
then for every $x\in\Ig$ and $b\in\Bg$ we have $x^*b^*\in\Ig$ hence $bx=(x^*b^*)^*\in\Ig$,
and thus $\Bg\Ig\subseteq\Ig$ as well.
\end{enumerate}
\end{remark}

In connection with the following observation we emphasize
that the ideal $\Ig$ of the Banach algebra $\Bg$ is \textit{not} assumed to be closed.

\begin{lemma}\label{exp_id}
Let $\Bg$ be a real associative unital Banach algebra.
If $\Ig$ is a one-sided ideal of $\Bg$ and $x\in\Bg$,
then 
$$x\in\Ig \iff(\forall t\in\RR)\quad \exp(tx)\in(\1+\Ig)\cap\Bg^\times. $$
\end{lemma}

\begin{proof} Assume that we have for instance $\Ig\Bg\subseteq\Ig$.
If $x\in\Ig$ and $t\in\RR$ then $\exp(tx)\exp(-tx)=\exp(-tx)\exp(tx)=\1$,
hence $\exp(tx)\in\Bg^\times$.
Moreover
$$\exp(tx)=\1+tx+\frac{t^2}{2}x+\cdots\in\1+x\Bg\subseteq\1+\Ig\Bg\subseteq\1+\Ig.$$
Conversely, assume that $\exp(tx)\in\1+\Ig$ for every $t\in\RR$.
Since $\lim\limits_{t\to 0}\Vert\1-\exp(tx)\Vert=0$,
there exists $t_0\in\RR\setminus\{0\}$ such that $\Vert\1-\exp(t_0 x)\Vert<1$.
Recall that for every $y\in\Bg$ with $\Vert\1-y\Vert<1$, 
the series
$$\log(\1-y):=-\sum_{k=1}^\infty\frac{1}{k}y^k\in y\Bg $$
is uniformly convergent
and for $y=\1-\exp(t_0 x)$ we have $\log(\1-y)=t_0x$ 
(see \cite[Lemma 2.1]{Up85}).
Then $x\in y\Bg=(\1-\exp(t_0 x))\Bg$
and the hypothesis $\exp(t_0x)\in\1+\Ig$ implies $x\in\Ig\Bg=\Ig$. 
\end{proof}

\begin{theorem}\label{alg_indeed}
Let $\Bg$ be a real associative unital Banach algebra with a one-sided ideal $\Ig$.
If $G_{\Ig}$ is an $\Ig$-restricted algebraic group in $\Bg$,
then its Lie algebra is a Lie subalgebra of $\Ig$ and can be described as
$$\gg_{\Ig}=\{x\in \Bg\mid(\forall t\in\RR)\quad \exp(tx)\in G_{\Ig}\}.$$
\end{theorem}

\begin{proof}
Recall from Definition~\ref{algebraic_restr} that there is an algebraic group $G$ in $\Bg$
with the Lie algebra $\gg$ such that $G_{\Ig}=G\cap(\1+\Ig)$ and $\gg_{\Ig}=\gg\cap\Ig$.
Since $\gg$ is a Lie subalgebra of $\Bg$ by Remark~\ref{closed_alg},
it then follows that $\gg_{\Ig}$ is a Lie subalgebra of $\Ig$
(see also Remark~\ref{simple}\eqref{simple_item3}).
Moreover, the description of $\gg_{\Ig}$ follows from Lemma~\ref{exp_id}.
\end{proof}







\bigskip
\begin{thebibliography}{100000000}



\bibitem[ALR10]{ALR10}
E.~Andruchow, G.~Larotonda, L.~Recht, 
Finsler geometry and actions of the $p$-Schatten unitary groups. 
\textit{Trans. Amer. Math. Soc.} \textbf{362} (2010), no.~1, 319--344.

\bibitem[AV07]{AV07}
E.~Andruchow, A.~Varela, 
Non positively curved metric in the space of positive definite infinite matrices. 
\textit{Rev. Un. Mat. Argentina} \textbf{48} (2007), no.~1, 7--15. 

\bibitem[BP66]{BP66}
V.K.~Balachandran, P.R.~Parthasarathy,
Cartan subalgebras of an $L^\ast$-algebra.
\textit{Math. Ann.} \textbf{166} (1966), 300--301.

\bibitem[Be06]{Be06}
D.~Belti\c t\u a,
\textit{Smooth Homogeneous Structures in Operator Theory}.
Chapman \& Hall/CRC Monographs and Surveys in Pure and Applied Mathematics, 137.
Chapman \& Hall/CRC, Boca Raton, FL, 2006.

\bibitem[Be11]{Be11}
D.~Belti\c t\u a,
Functional analytic background for a theory of infinite-dimensional reductive Lie groups.
In: K.-H. Neeb, A. Pianzola (eds.),
\textit{Developments and Trends in Infinite-Dimensional Lie Theory}, Progr. Math., 288,
Birkh\"auser Boston, Inc., Boston, MA, 2011, pp.~367--392.

\bibitem[BPW14a]{BPW13a}
D.~Belti\c t\u a, S.~Patnaik, G.~Weiss,
$B(H)$-Commutators: A Historical Survey II  and
recent advances on
commutators of compact operators.
In: D. Gaspar, D. Timotin, F.-H. Vasilescu, and L. Zsid\'o (eds.), 
\textit{The Varied Landscape of Operator Theory} 
(Conference Proceedings, Timisoara, July 2--7, 2012), Theta, Bucharest, 2014, pp. 57--75. 

\bibitem[BPW14b]{BPW13b}
D.~Belti\c t\u a, S.~Patnaik, G.~Weiss,
Interplay between algebraic groups, Lie algebras and operator ideals.
In: D. Gaspar, D. Timotin, F.-H. Vasilescu, and L. Zsid\'o (eds.), 
\textit{The Varied Landscape of Operator Theory} (Conference Proceedings, Timisoara, July 2--7, 2012), 
Theta, Bucharest, 2014, pp. 77--97. 


\bibitem[BP07]{BP07}
D.~Belti\c t\u a, B.~Prunaru,
Amenability, completely bounded projections, dynamical systems and smooth orbits.
\textit{Integral Equations Operator Theory} \textbf{57} (2007), no. 1, 1--17.

\bibitem[BRT07]{BRT07}
D.~Belti\c t\u a, T.S.~Ratiu, A.B.~Tumpach,
The restricted Grassmannian, Banach Lie-Poisson spaces, and coadjoint orbits.
\textit{J. Funct. Anal.} \textbf{247} (2007), no.~1, 138--168.

\bibitem[Boy80]{Bo80}
R.P.~Boyer, 
Representation theory of the Hilbert-Lie group ${\rm U}(H)\sb{2}$. 
\textit{Duke Math. J.} \textbf{47} (1980), no. 2, 325--344. 

\bibitem[CD13]{CD13}
E.~Chiumiento, M.E.~Di Iorio y Lucero, 
Geometry of unitary orbits of pinching operators. 
\textit{J. Math. Anal. Appl.} \textbf{402} (2013), no.~1, 103--118.

\bibitem[Da96]{Da96}
K.R.~Davidson,
\textit{$C^*$-algebras by Example}.
Fields Institute Monographs, 6. American Mathematical Society, Providence, RI, 1996.

\bibitem[DPW02]{DPW02}
I.~Dimitrov, I.~Penkov, J.A.~Wolf, 
A Bott-Borel-Weil theory for direct limits of algebraic groups. 
\textit{Amer. J. Math.} \textbf{124} (2002), no.~5, 955--998. 

\bibitem[Di54]{Di54}
J.~Dixmier, 
Sous-anneaux ab\'eliens maximaux dans les facteurs de type fini. 
\textit{Ann. of Math. (2)} \textbf{59} (1954), 279--286. 

\bibitem[Di69]{Di69}
J.~Dixmier,
\textit{Les alg\`ebres d'op\'erateurs dans l'espace hilbertien (alg\`ebres de von Neumann)}.
Deuxi\`eme \'edition, revue et augment\'ee.
Cahiers Scientifiques, Fasc. XXV.
Gauthier-Villars \'Editeur, Paris, 1969.

\bibitem[DFWW04]{DFWW04}
K.~Dykema, T.~Figiel, G.~Weiss, M.~Wodzicki,
Commutator structure of operator ideals.
\textit{Adv. Math.} \textbf{185} (2004), no.~1, 1--79.

\bibitem[FM77]{FM77}
J.~Feldman, C.C.~Moore, 
Ergodic equivalence relations, cohomology, and von Neumann algebras. I,II. 
\textit{Trans. Amer. Math. Soc.} \textbf{234} (1977), no.~2, 289--324,  325--359. 

\bibitem[GK69]{GK69}
I.C.~Gohberg, M.G.~Kre\u\i n,
\textit{Introduction to the Theory of Linear Nonselfadjoint Operators}.
Translations of Mathematical Monographs, Vol. 18,
American Mathematical Society, Providence, R.I., 1969.

\bibitem[GW09]{GW09}
R.~Goodman, N.R.~Wallach,
\textit{Symmetry, representations, and invariants}.
Graduate Texts in Mathematics, 255. Springer, Dordrecht, 2009.

\bibitem[Ha82]{Ha82} 
P.R.~Halmos, 
\textit{A Hilbert space problem book}. Second edition. 
Graduate Texts in Mathematics, 19. Encyclopedia of Mathematics and its Applications, 17. 
Springer-Verlag, New York-Berlin, 1982. 

\bibitem[dlH72]{dlH72}
P.~de la Harpe,
\textit{Classical Banach-Lie Algebras and Banach-Lie Groups of Operators in Hilbert Space}.
Lecture Notes in Mathematics, Vol. 285.
Springer-Verlag, Berlin-New York, 1972.

\bibitem[HK77]{HK77}
L.A.~Harris, W.~Kaup,
Linear algebraic groups in infinite dimensions.
\textit{Illinois J. Math.} \textbf{21} (1977), no.~3, 666--674.

\bibitem[Hi85]{Hi85}
A.~Hinkkanen,
On the diagonalization of a certain class of operators.
\textit{Michigan Math. J.} \textbf{32} (1985), no.3, 349--359.

\bibitem[KW12]{KW12}
V.~Kaftal, G.~Weiss,
Majorization and arithmetic mean ideals.
\textit{Indiana Univ. Math. J.}  \textbf{60} (2011),  1393--1424. 

\bibitem[Kn02]{Kn02}
Knapp, A.~W., 
\textit{Lie Groups Beyond an Introduction}. Second edition. 
Progress in Mathematics, 140. Birkh\"auser Boston, Inc., Boston, MA, 2002. 

\bibitem[KMRT98]{KMRT98} 
M.-A.~Knus, A.~Merkurjev, M.~Rost, J.-P.~Tignol,
\textit{The book of involutions}.
American Mathematical Society Colloquium Publications, 44.
American Mathematical Society, Providence, RI, 1998.

\bibitem[Ku58]{Ku58}
S.T.~Kuroda,
On a theorem of Weyl-von Neumann.
\textit{Proc. Japan Acad.} \textbf{34} (1958), 11--15.


\bibitem[NRW01]{NRW01}
L.~Natarajan, E.~Rodriguez-Carrington, J.A.~Wolf, 
The Bott-Borel-Weil theorem for direct limit groups. 
\textit{Trans. Amer. Math. Soc.} \textbf{353} (2001), no.~11, 4583--4622.

\bibitem[Ne98]{Ne98} 
K.-H.~Neeb, 
Holomorphic highest weight representations of infinite-dimensional complex classical groups. 
\textit{J. reine angew. Math.} \textbf{497} (1998), 171--222.

\bibitem[Ne02]{Ne02} 
K.-H.~Neeb, 
Classical Hilbert-Lie groups, their extensions and their homotopy groups. 
In: A.~Strasburger, J.~Hilgert, K.-H.~Neeb and W.~Wojty\'nski (eds.), 
\textit{Geometry and Analysis on Finite- and Infinite-Dimensional Lie Groups} (B\c edlewo, 2000), 
Banach Center Publ., 55, Polish Acad. Sci., Warsaw, 2002, pp.~87--151.

\bibitem[Ne04]{Ne04} 
K.-H.~Neeb, 
Infinite-dimensional groups and their representations. 
in: J.-Ph.~Anker, B.\O rsted (eds.), \textit{Lie Theory}, 
Progr. Math., 228, Birkh\"auser Boston, Boston, MA, 2004, pp.~213--328. 

\bibitem[Ne06]{Ne06}
K.-H.~Neeb, 
Towards a Lie theory of locally convex groups. 
\textit{Jpn. J. Math.} \textbf{1} (2006), no. 2, 291--468. 

\bibitem[NP03]{NP03}
K.-H.~Neeb, I.~Penkov,
Cartan subalgebras of $\mathfrak{gl}_\infty$.
\textit{Canad. Math. Bull.} \textbf{46} (2003), no.~4, 597--616.

\bibitem[OP10]{OP10}
N.~Ozawa, S.~Popa, 
On a class of ${\rm II}_1$ factors with at most one Cartan subalgebra. 
\textit{Ann. of Math. (2)} \textbf{172} (2010),  no.~1, 713--749.

\bibitem[Par85]{Par85}
J.A.~Packer, 
Point spectrum of ergodic abelian group actions and the corresponding group-measure factors. 
\textit{Pacific J. Math.} \textbf{119} (1985), no.~2, 381--405. 

\bibitem[Pau02]{Pa02}
V.~Paulsen,
\textit{Completely Bounded Maps and Operator Algebras}.
Cambridge Stud. Adv. Math., vol. 78, Cambridge
Univ. Press, Cambridge, 2002.

\bibitem[Po83]{Po83}
S.~Popa, 
Singular maximal abelian $*$-subalgebras in continuous von Neumann algebras. 
\textit{J. Funct. Anal.} \textbf{50} (1983), no.~2, 151--166. 

\bibitem[Po90]{Po90}
S.~Popa, 
Some rigidity results in type ${\rm II}_1$ factors. 
\textit{C. R. Acad. Sci. Paris S\'er. I Math.} \textbf{311} (1990),  no.~9, 535--538. 

\bibitem[Re08]{Re08}
J.~Renault,
Cartan subalgebras in $C^*$-algebras.
\textit{Irish Math. Soc. Bull.} \textbf{61} (2008), 29--63.

\bibitem[Sch60]{Sch60}
J.R.~Schue,
Hilbert space methods in the theory of Lie algebras.
\textit{Trans. Amer. Math. Soc.} \textbf{95} (1960), 69--80.

\bibitem[Sch61]{Sch61}
J.R.~Schue,
Cartan decompositions for $L^\ast$ algebras.
\textit{Trans. Amer. Math. Soc.} \textbf{98} (1961), 334--349.

\bibitem[SS08]{SS08}
A.M.~Sinclair, R.R.~Smith,
\textit{Finite von Neumann Algebras and Masas}.
London Mathematical Society Lecture Note Series, 351.
Cambridge University Press, Cambridge, 2008.

\bibitem[St75]{St75}
I.~Stewart,
\textit{Lie Algebras Generated by Finite-Dimensional Ideals}.
Research Notes in Mathematics, Vol. 2.
Pitman Publishing, London-San Francisco, Calif.-Melbourne, 1975.

\bibitem[Up85]{Up85}
H.~Upmeier,
\textit{Symmetric Banach Manifolds and Jordan $C^\ast$-algebras}.
North-Holland Mathematics Studies, 104.
Notas de Matem\'atica, 96.
North-Holland Publishing Co., Amsterdam, 1985.

\bibitem[Vo96]{Vo96}
D.~Voiculescu, 
The analogues of entropy and of Fisher's information measure in free probability theory. III. 
The absence of Cartan subalgebras. 
\textit{Geom. Funct. Anal.} \textbf{6} (1996),  no.~1, 172�199.

\bibitem[Wo05]{Wo05}
J.A.~Wolf, 
Principal series representations of direct limit groups. 
\textit{Compos. Math.} \textbf{141} (2005), no.~6, 1504--1530.

\end{thebibliography}
\end{document}